\documentclass[letterpaper, 10 pt, conference]{ieeeconf}  
\usepackage{booktabs}
\usepackage{lineno}
\usepackage[cmex10]{amsmath}
\usepackage{amscd,amsfonts,amsmath,amssymb,mathrsfs,pifont,stmaryrd,tipa}
\usepackage{array,cases,dsfont,graphicx,texdraw}
\usepackage{wrapfig}
\usepackage{graphics} 
\usepackage{epsfig} 
\usepackage{stfloats}
\usepackage{subfig}
\usepackage{hyperref}
\usepackage{color}

\newtheorem{Remark}{Remark}

\newtheorem{Problem}{Problem}

\newtheorem{Theorem}{Theorem}
\newtheorem{Lemma}{Lemma}

\newtheorem{Assumption}{Assumption}

\input{tcilatex}                                                          

\IEEEoverridecommandlockouts                              
\title{\LARGE \bf
Differentially private Nash equilibrium seeking for networked aggregative games}

\author{Maojiao Ye, Guoqiang Hu, Lihua Xie and Shengyuan Xu
\thanks{M. Ye and S. Xu are with the School of Automation, Nanjing University of Science and Technology, Nanjing 210094, P.R. China (Email: mjye@njust.edu.cn; syxu@njust.edu.cn); G. Hu and L. Xie are with the School of Electrical and Electronic Engineering, Nanyang Technological University, 639798, Singapore (Email: gqhu@ntu.edu.sg; ELHXIE@ntu.edu.sg).}
\thanks{This work is supported by the Natural Science Foundation of China (NSFC) under Grant 61803202, the Natural Science Foundation of Jiangsu Province, No. BK20180455 and the Fundamental Research Funds for the Central Universities, No. 30918011332.}
}
\begin{document}

\maketitle
\thispagestyle{empty}
\pagestyle{empty}

\begin{abstract}
This paper considers the privacy-preserving Nash equilibrium seeking strategy design for a class of networked aggregative games, in which the players' objective functions are considered to be sensitive information to be protected. In particular, we consider that the networked game is free of central node and the aggregate information is not directly available to the players. As there is no central authority to provide the aggregate information required by each player to update their actions, a dynamic average consensus protocol is employed to estimate it. To protect the players' privacy, we perturb the transmitted information among the players by independent random noises drawn from Laplace distributions. By synthesizing the perturbed average consensus protocol with a gradient algorithm, distributed privacy-preserving Nash equilibrium seeking strategies are established for the aggregative games under both fixed and time-varying communication topologies. With explicit quantifications of the mean square errors, the convergence results of the proposed methods are presented. Moreover, it is analytically proven that the proposed algorithm is $\epsilon$-differentially private, where $\epsilon$ depends on the stepsize of the gradient algorithm and the scaling parameter of the random variables. The presented results indicate that there is a tradeoff between the convergence accuracy and the privacy level.  Lastly, a numerical example is provided for the verification of the proposed methods.
\end{abstract}


\section{INTRODUCTION}

Privacy has become a critical concern for many practical systems that involve sensitive data transmission and collection. Wireless sensor networks \cite{LiAHN}, smart grids \cite{LiangTSG13}, social networks \cite{Wu10}, just to name a few, are typical examples that are in urgent need of privacy protection techniques. Inspired by the fact that privacy preservation is pivotal in information-sensitive systems, privacy protection methods have gained increasing attention in recent years. For instance, in the field of database and data mining, cryptographic secure multi-party computation methods, random perturbation techniques and $l$-diversity, $k$-anonymity based algorithms were adopted for information sharing systems, data collection systems and data publishing systems, respectively \cite{LiAHN}. Homomorphic encryption was adapted for privacy protection in smart grids in \cite{LiangTSG13}.  Randomization, $k$-anonymity and generalization-based approaches can be utilized for privacy protection in social networks \cite{Wu10}. Motivated by the importance of privacy protection, this paper aims to achieve the distributed Nash equilibrium seeking for networked aggregative games with privacy guarantees.

In many practical situations, the utilities associated with the interacting decision-makers rely both on the decision-maker's own action and an aggregate of all the decision-makers' actions. For example, in the energy consumption model described in \cite{YEcyber17}, the utilities of the electricity users are determined by the user's own energy consumption as well as the total energy consumption in the electricity market. Cournot price/quantity competitions among multiple oligopolistic firms fall into similar scenario \cite{Agiza03}. In  factory production, the utility of each manufactory relies on the averaged output of all the engaged manufactories and the manufactory's own production \cite{WangIET}. In addition, public good provision models and many other examples are also of aggregative nature, i.e., the interacting participants affect each other through a specific aggregate of their actions rather than in an arbitrary fashion \cite{Cornes}. Aggregative games serve as powerful game theoretic models to accommodate these competitive circumstances with aggregative interactions among multiple decision-makers. Motivated by the wide applications of aggregative games in distributed systems, Nash equilibrium seeking for aggregative games on communication graphs is attracting increasing attention in recent years \cite{KoshalOR16}-\cite{PariseCDC15}.

A discrete-time method was proposed in \cite{KoshalOR16} for networked aggregative games and a continuous-time counterpart was provided in \cite{YEcyber17} considering their applications for demand response in smart grids. A  gossip-based algorithm was proposed to achieve distributed Nash equilibrium seeking in an asynchronous fashion in \cite{Salehisadaghiani}. Coupled constraints were further addressed in \cite{LiangAT17}. The authors in \cite{PariseCDC15} considered  quadratic quasi-aggregative games, in which the players' objective functions depend on the players' actions and an aggregate of its neighbors' actions. From the perspective of privacy issues, the works in \cite{Cummings}\cite{ZhouJSA} shed some light on the privacy protection for aggregative games. However, in the mechanisms of the game, there is a mediator/weak mediator, that can receive information from the players and give suggested actions for the players, to induce the players' behaviours. Hence, the methods in \cite{Cummings}\cite{ZhouJSA} are not distributed. To achieve distributed Nash equilibrium seeking for networked games, the players usually need to broadcast their local information to their neighbors via local communication networks. The information dissemination among the players may raise privacy concern for the players. \textbf{Nevertheless, to the best of the authors' knowledge, privacy  issues have rarely been explored for distributed Nash equilibrium seeking schemes though it is a problem of significant interest.} Motivated by the above observation, this paper considers privacy-preserving distributed Nash equilibrium seeking for networked aggregative games by utilizing the notion of differential privacy \cite{Dwork}.

Differential privacy has been widely adopted to describe the privacy level in many situations that include sensitive information. For example, the differential privacy of the agents engaged in distributed optimization problems was established in \cite{Huang15}-\cite{DingCDC18}. In \cite{Huang15}, random noises were utilized for the protection of the agents' objective functions. Based on the notion of differential privacy, the tradeoff between the convergence accuracy and the privacy level was analyzed. In \cite{HanTAC17}, the constraints of the optimization problem were considered to be the private information to be protected. Based on the stochastic gradient method, a private optimization algorithm was developed by introducing additive noises.  Instead of probing the transmitted messages, the authors in \cite{Nozari} proposed a functional perturbation algorithm which ensures that the inaccuracy of the optimization algorithm is only resulted from the introduced noises. In \cite{DingCDC18}, both the agents' states and moving directions were masked by random noises to achieve the privacy protection of the agents' local objective functions. Differentially private average consensus protocols have also been widely studied (see e.g., \cite{MoTAC17}-\cite{Manitara13}, to mention just a few). For instance, the agents' initial states were considered to be private information to be protected in the average consensus problems in \cite{MoTAC17}, where the authors developed an asymptotically convergent algorithm to achieve privacy protection. Both a non-exact convergent algorithm and an almost surely convergent algorithm were established to achieve differentially private average consensus in \cite{NozariAT}. Moreover, differential privacy has been extensively investigated in the area of database and data mining. Interested readers are referred to \cite{Dwork} for a survey on differential privacy.

To shed some light on privacy-protected Nash equilibrium seeking for aggregative games under distributed communication networks, this paper proposes a distributed algorithm by perturbing the transmitted information among the players using random variables. Moreover, motivated by the observation that differential privacy is robust against the auxiliary information exposed to the adversary, we adopt it to describe the players' privacy level in this paper. In brief, the main contributions of the paper are summarized as follows.
\begin{enumerate}
  \item This paper considers privacy-preserving distributed Nash equilibrium seeking for networked aggregative games.  To achieve the goal, we employ a dynamic average consensus protocol for the distributed estimation of the players' aggregate actions. To protect the players' privacy, the information transmitted among the players in the consensus part is perturbed by independent random noises drawn from Laplace distributions. By utilizing the perturbed consensus protocol and a gradient method with a decaying stepsize, distributed privacy-preserving Nash equilibrium strategies are established for games under fixed and time-varying communication graphs, respectively.
  \item The convergence result and the privacy level of the proposed methods are analytically investigated. In particular, the mean square error bound and privacy parameter are explicitly quantified. The presented results illustrate that the privacy level depends on the selection of the initial stepsize, the decaying rate of the stepsize as well as the scaling parameter of the random variables.  Moreover,  it is shown that there is a tradeoff between the convergence accuracy and the privacy level.
\end{enumerate}

The rest of the paper is organized as follows. Section \ref{NP} provides some preliminaries and formulates the considered problem. The main results are given in Section \ref{main}, where the distributed algorithm is presented with its convergence accuracy and privacy level successively investigated. Moreover, the results under fixed communication topologies are extended to time-varying communication topologies in Section \ref{time}.  In Section \ref{num_ex}, a numerical example is provided to verify the effectiveness of the proposed method. Lastly, conclusions are drawn in Section \ref{conc}.

\emph{Notations:} In this paper, we use $\mathbb{R}$ and $\mathbb{R}_{+}$ to denote the set of real numbers and positive real numbers, respectively.  Moreover, $\mathbb{Z}^+$ denotes the set of non-negative integers. Let $v$ be a vector or matrix, then $||v||$ denotes the $\ell_2$-norm of $v$. Moreover, $||v||_{\infty}$ denotes the $\ell_{\infty}$-norm of $v$. We say that a random variable $w\sim Lap(b)$, where $b\in \mathbb{R}_{+}$, if its probability density function is
\begin{equation}\nonumber
\mathcal{L}(w,b)=\frac{1}{2b}e^{-\frac{|w|}{b}}.
\end{equation}
In addition, $\mathbb{E}(\Phi)$ is the expectation of $\Phi.$ The partial derivative of $f_i(\mathbf{x})$ with respect to $x_i$ is denoted as $\nabla_i f_i(\mathbf{x}),$ where $\mathbf{x}=[x_1,x_2,\cdots,x_N]^T.$ Moreover, $\mathbf{\mathbf{1}}_N$ denotes an $N$-dimensional column vector whose elements are all $1$ and the transpose of $\Psi$, where $\Psi$ is either a matrix or a vector, is denoted as $\Psi^T.$ The notation $[g_i]_{vec}$ for $i\in\{1,2,\cdots,N\}$ denotes a column vector whose $i$th element is $g_i.$ For a sequence $g(k)$, where $k\in\{1,2,3,\cdots\}$ and $g(k)$ is either a vector or a scalar, $g(\infty)=\lim_{k\rightarrow \infty}g(k)$ given that $\lim_{k\rightarrow \infty}g(k)$ exists. The maximum and minimum eigenvalues of a symmetric matrix $P$ are denoted as $\lambda_{max}(P)$ and $\lambda_{min}(P),$ respectively. For a matrix $Q$, $[Q]_{ij}$ denotes the entry on the $i$th row and $j$th column of $Q$. In addition, for $l_i\in\mathbb{R},i\in\{1,2,\cdots,N\},$ the maximum and minimum values of $l_i$ are denoted as $\max_{i\in\{1,2,\cdots,N\}}\{l_i\}$ and $\min_{\{1,2,\cdots,N\}}\{l_i\},$ respectively.

\section{Preliminaries and Problem Formulation}\label{NP}
\subsection{Preliminaries}
In the following, we provide some preliminaries on game theory and differential privacy.

\subsubsection{Game Theory}
The following game related definitions are adopted from \cite{YEcyber17}.
\begin{definition}
\emph{(A Normal Form Game)} A game in a normal form is defined as a triple $\Gamma=\{\mathcal{V},X,f\},$ where $\mathcal{V}=\{1,2,\cdots,N\}$ is the set of players, $X=X_1\times X_2\cdots \times X_N, X_i\subseteq \mathbb{R}$ is the action set of player $i$ and $f=(f_1,f_2,\cdots,f_N),$ where $f_i$ is the cost function of player $i$.
\end{definition}
\begin{definition}
\emph{(Nash Equilibrium)} Nash equilibrium is an action profile on which no player can reduce its cost by unilaterally changing its own action, i.e., an action profile $\mathbf{x}^*=(x_i^*,\mathbf{x}_{-i}^*)\in X$ is a Nash equilibrium if for $i\in\mathcal{V},$
\begin{equation}
f_i(x_i^*,\mathbf{x}_{-i}^*)\leq f_i(x_i,\mathbf{x}_{-i}^*),
\end{equation}
for $x_i\in X_i,$ where $\mathbf{x}_{-i}=[x_1,x_2,\cdots,x_{i-1},x_{i+1},\cdots,x_N]^T.$
\end{definition}

\begin{definition}
(\emph{Aggregative Game}) A game $\Gamma$ is aggregative if there exists an aggregative function  $l(\mathbf{x}):X\in \mathbb{R}$, which is  continuous, additive and separable, such that there are functions $\tilde{f}_i(x_i,l(\mathbf{x})),i\in\mathcal{V}$ that satisfy
\begin{equation}
\tilde{f}_i(x_i,l(\mathbf{x}))=f_i(x_i,\mathbf{x}_{-i}),\forall \mathbf{x}\in X.
\end{equation}
\end{definition}

Without loss of generality, we consider that $l(\mathbf{x})=\frac{1}{N}\sum_{i=1}^N x_i$ in this paper.

\subsubsection{Differential privacy}
Differential privacy serves as a mathematical quantification on the level of the engaged individuals' privacy guarantee in a statistical database.  It provides a rigorous and formal mathematical formulation on the privacy of the sensitive data. We refer readers to \cite{Dwork} for more detailed elaborations on differential privacy and give the subsequent definitions for clarity of presentation.

\begin{definition}
\emph{(Adjacency)} Two function sets $F^{(1)}=\{f_i^{(1)}\}_{i=1}^N$, $F^{(2)}=\{f_i^{(2)}\}_{i=1}^N$ are said to be adjacent if there exists some $i_0\in\mathcal{V}$ such that $f_i^{(1)}=f_i^{(2)},\forall i\neq i_0,$
and $f_{i_0}^{(1)}\neq f_{i_0}^{(2)}$ \cite{Huang15}\cite{DingCDC18}.
\end{definition}

\begin{definition}
\emph{($\epsilon$-Differential Privacy)} Given a positive constant $\epsilon,$ adjacent function sets $F^{(1)}$, $F^{(2)}$ and any observation $\mathcal{O},$ the algorithm is $\epsilon$-differentially private if
\begin{equation}
\mathbb{P}\{F^{(1)}|\mathcal{O}\}\leq e^{\epsilon} \mathbb{P}\{F^{(2)}|\mathcal{O}\},
\end{equation}
where $\mathbb{P}\{F^{(j)}|\mathcal{O}\}$ for $j\in\{1,2\}$ is the conditional probability representing the probability of inferring $F^{(j)}$ from the observation $\mathcal{O}$ \cite{Huang15}\cite{DingCDC18}.
\end{definition}

\begin{Remark}
The above defined $\epsilon$-differential privacy illustrates that from the sequences of observations, the adversary could not distinguish between the two function sets with a high probability. Hence, it is challenging for the adversary to identify the players' sensitive information, which further indicates that the players are protected from information leakage. Note that a smaller $\epsilon$ indicates a higher level privacy.
\end{Remark}

\subsection{Problem formulation}
\begin{Problem}
Consider an aggregative game in which player $i$ intends to
\begin{equation}\label{eqform}
\text{min}_{x_i}\ \ f_i(x_i,\mathbf{x}_{-i})
\end{equation}
where
\begin{equation}
f_i(x_i,\mathbf{x}_{-i})=\tilde{f}_i(x_i,\bar{x}),
\end{equation}
and $\bar{x}=\frac{1}{N}\sum_{j=1}^N x_j$
for $i\in\mathcal{V}.$ Suppose that there is no central authority to broadcast $\bar{x}$ and the players can communicate with each other via a communication graph $\mathcal{G},$ the objective of this paper is to design a Nash equilibrium seeking strategy for the aggregative game such that
\begin{enumerate}
  \item Given any positive constant $\epsilon$, the strategy can be $\epsilon$-differentially private by tuning the control parameters.
  \item The players' actions can be driven to a neighborhood of the Nash equilibrium point in the mean square sense, $\lim_{k\rightarrow\infty}\mathbb{E}(||\mathbf{x}(k)-\mathbf{x}^*||^2)\leq D$, where $D$ is a positive constant that is as small as possible;
\end{enumerate}
\end{Problem}

\begin{Remark}
Note that the considered aggregative games are practically inspired by the fact that in many decision-making processes (e.g., energy consumption control, Cournot quantity competitions, public good provision, factory production, to mention just a few), the decision-makers affect the others via their averaged/aggregate behaviors. Moreover, different from the previous works on distributed Nash equilibrium seeking for aggregative games \cite{KoshalOR16}-\cite{PariseCDC15}, \textbf{the objective of this paper includes privacy protection for the players' cost functions $f_i(x_i,\mathbf{x}_{-i}),i\in\mathcal{V}$.}
\end{Remark}

For notational convenience, let $\mathbf{x}=[x_1,x_2,\cdots,x_N]^T$ and with a slight abuse of notation, $f_i(x_i,\mathbf{x}_{-i})$ and $\nabla_i f_i(x_i,\mathbf{x}_{-i})$ might be written as $f_i(\mathbf{x})$ and $\nabla_i f_i(\mathbf{x})$, respectively in the rest of the paper.

The following assumptions (see e.g., \cite{YETAC17}) will be utilized to establish the main results of the paper.

\begin{Assumption}\label{a2}
The players' objective functions are twice continuously differentiable functions and $\nabla_i f_i(\mathbf{x})$ is globally Lipschitz for $i\in\mathcal{V}$, i.e., there exists a positive constant $l_i$ such that
\begin{equation}
||\nabla_i f_i(\mathbf{x})-\nabla_i f_i(\mathbf{y})||\leq l_i||\mathbf{x}-\mathbf{y}||,
\end{equation}
for $\mathbf{x},\mathbf{y}\in \mathbb{R}^N,i\in\mathcal{V}.$
\end{Assumption}

\begin{Assumption}\label{a3}
There exists a positive constant $m$ such that for $\mathbf{x},\mathbf{y}\in \mathbb{R}^N,$
\begin{equation}
(\mathbf{x}-\mathbf{y})^T (g(\mathbf{x})-g(\mathbf{y}))\geq m||\mathbf{x}-\mathbf{y}||^2,
\end{equation}
where $g(\mathbf{x})=\left[\nabla_1 f_1(\mathbf{x}),\nabla_2 f_2(\mathbf{x}),\cdots,\nabla_N f_N(\mathbf{x})\right]^T.$
\end{Assumption}

\begin{Remark}
The strong monotonicity condition in Assumption \ref{a3} characterizes a global Nash equilibrium, i.e., the Nash equilibrium is unique under Assumption \ref{a3}. In addition, by Assumption \ref{a3},
\begin{equation}
g(\mathbf{x})=\mathbf{0}_N,
\end{equation}
if and only if $\mathbf{x}=\mathbf{x}^*.$
Note that it is a widely adopted assumption in the existing literature and we refer interested readers to  \cite{Facchinei03} for more insights on the assumption.
\end{Remark}

\begin{Assumption}\label{ass}
$\nabla_i f_i(\mathbf{x})$ for $i\in\mathcal{V}$ are uniformly bounded, i.e., there exists a positive constant $C$ such that for $\mathbf{x}\in \mathbb{R}^N,$
$|\nabla_i f_i(\mathbf{x})|\leq C,$ $\forall i\in\mathcal{V}.$
\end{Assumption}

\section{Privacy-preserving Nash equilibrium seeking under fixed communication graphs}\label{main}
In this section, a privacy-preserving distributed Nash equilibrium seeking strategy will be proposed for games under a fixed undirected communication graph $\mathcal{G}$, which is defined as $\mathcal{G}=\{\mathcal{V},\mathcal{E}\}.$ Moreover, $\mathcal{V}=\{1,2,\cdots,N\}$ denotes the set of vertices and $\mathcal{E}\subseteq \mathcal{V}\times \mathcal{V}$ is the set of edges. In this way, player $j$ can communicate with player $i$ if and only if $(j,i)\in \mathcal{E}.$  Associate with $\mathcal{G}$ a weight matrix $\mathcal{A}=[a_{ij}]$ whose element on the $i$th row and $j$th column is $a_{ij}$. In this section,  we suppose that $\mathcal{G}$ and $\mathcal{A}$ satisfy the following assumption:
\begin{Assumption}\label{comm}
The communication graph $\mathcal{G}$ is undirected and connected. Moreover, $\mathcal{A}$ satisfies
 \begin{enumerate}
   \item $a_{ii}>0,$ $a_{ij}>0$ if $(j,i)\in \mathcal{E}$ and $a_{ij}=0$ if $(j,i)\notin \mathcal{E}.$
   \item $\mathcal{A}$ is doubly stochastic, i.e., $\mathcal{A}\mathbf{1}_N=\mathbf{1}_N$ and $\mathbf{1}_N^T\mathcal{A}=\mathbf{1}_N^T$.
 \end{enumerate}
\end{Assumption}

Let $\{\lambda_i\}_{i=1}^N$ be the eigenvalues of $\mathcal{A}$ and suppose that $\lambda_1\geq \lambda_2\geq \cdots, \geq \lambda_N.$ Then, by Assumption \ref{comm}, $\lambda_1=1,$ $\lambda_2<1$, $\lambda_N>-1$.
Moreover, the following lemma holds.
\begin{Lemma}\label{lemam1}
\cite{Xiao}\cite{Johansson} Suppose that Assumption \ref{comm} is satisfied. Then,
\begin{equation}
\left|\left|\left(\mathcal{A}^k-\frac{\mathbf{1}_N\mathbf{1}_N^T}{N}\right)\right|\right|=\left|\left|\left(\mathcal{A}-\frac{\mathbf{1}_N\mathbf{1}_N^T}{N}\right)^k\right|\right|
\leq \gamma^k,
\end{equation}
where $\gamma$ is a positive constant such that
\begin{equation}
\rho\left(\mathcal{A}-\frac{\mathbf{1}_N\mathbf{1}_N^T}{N}\right)\leq \gamma <1,
\end{equation}
and $\rho\left(\mathcal{A}-\frac{\mathbf{1}_N\mathbf{1}_N^T}{N}\right)$ is the spectral radius of $\mathcal{A}-\frac{\mathbf{1}_N\mathbf{1}_N^T}{N}.$
\end{Lemma}

\subsection{Method development}
As there is no central authority to broadcast $\bar{x}$ to the players, we design a strategy by utilizing a synthesis of consensus protocols and optimization algorithms as in \cite{KoshalOR16}-\cite{YEcyber17}. Moreover, to protect the players' privacy, we utilize independent random variables to perturb the transmitted information among the players in the average consensus protocol. More specifically, each player $i,i\in\mathcal{V}$ can update its action according to
\begin{equation}\label{al1}
x_i(k+1)=x_i(k)-\alpha_kg_i(x_i(k),y_i(k)),
\end{equation}
where $k\in \mathbb{Z}^+$, $\alpha_k=cq^k$, $c\in \mathbb{R}_+,$ $q\in(0,1)$ and $g_i(x_i,y_i)=\left(\frac{\partial \tilde{f}_i(x_i,\bar{x})}{\partial x_i}+\frac{\partial \tilde{f}_i(x_i,\bar{x})}{\partial \bar{x}}\frac{\partial \bar{x}}{\partial x_i}\right)\left.\right|_{\bar{x}=y_i}$. Moreover, $y_i(k)$ is an intermediate variable updated according to
\begin{equation}\label{al2}
y_i(k+1)=\sum_{j=1}^{N}a_{ij}p_j(k)+x_i(k+1)-x_i(k),
\end{equation}
in which $y_i(0)=x_i(0)$, $p_i(k)=y_i(k)+w_i(k)$, $w_i(k)\sim Lap(\theta_{k})$ for $i\in\mathcal{V}$ are independent random variables, $\theta_{k}=d \bar{q}^k$, $d\in \mathbb{R}_+$ and $\bar{q}\in (q,1)$.
The steps of the designed algorithm are described as follows.\\
\ \\
\begin{tabular}{l}
\toprule
Privacy-preserving distributed Nash seeking:\\
\midrule
\textbf{Initialization:} Choose $x_i(0)\in \mathbb{R}$ and $y_i(0)= x_i(0)$.\\
\textbf{Iterations:}\\
      1. Define $p_i(k)=y_i(k)+w_i(k)$ \\
      2. Update $x_i(k)$ according to\\
         $\ \ \ \ x_i(k+1)=x_i(k)-\alpha_k g_i(x_i(k),y_i(k))$\\
      3. Update $y_i(k)$ according to\\
         $\ \ \ \ y_i(k+1)=\sum_{j=1}^{N}a_{ij}p_j(k)+x_i(k+1)-x_i(k)$\\
   \textbf{end}\\
\bottomrule
\end{tabular}
\\

Let $\mathbf{x}(k)=[x_1(k),x_2(k),\cdots,x_N(k)]^T,\mathbf{y}(k)=[y_1(k),y_2(k),\cdots,y_N(k)]^T,$ $\mathbf{w}(k)=[w_1(k),w_2(k),\cdots,w_N(k)]^T$ and $\mathbf{p}(k)=[p_1(k),p_2(k),\cdots,p_N(k)]^T.$
Then, the concatenated vector form of \eqref{al1}-\eqref{al2} is
\begin{equation}\label{cotn}
\begin{aligned}
\mathbf{x}(k+1)&=\mathbf{x}(k)-\alpha_k [g_i(x_i(k),y_i(k))]_{vec}\\
\mathbf{y}(k+1)&=\mathcal{A}\mathbf{p}(k)+\mathbf{x}(k+1)-\mathbf{x}(k).
\end{aligned}
\end{equation}

\begin{Remark}
In this paper, we suppose that there is no central authority to broadcast $\bar{x}(k)$ to the players. Hence, player $i,i\in\mathcal{V}$ would generate a local variable $y_i(k)$ to estimate $\bar{x}(k)$. Moreover, the update of $y_i(k)$ in \eqref{al2} is motivated by the dynamic average consensus protocol in \cite{YEcyber17}\cite{ZhuAT10}. However, in the dynamic average consensus protocol of \cite{YEcyber17}, the players communicate with their neighbors on their estimates of $\bar{x}(k),$ which may raise privacy concern. Hence, in \eqref{al2}, we perturb their transmitted information by utilizing independent random variables.  Moreover, different from the continuous-time scenario proposed in our previous work in \cite{YEcyber17}, we adopt a decaying stepsize in the presented algorithm to release the side-effect of the random noises on the convergence properties to some extent.
\end{Remark}

\begin{Remark}
In \cite{HeACC17}, the authors provided necessary and sufficient conditions for general noise adding mechanisms to achieve differential privacy. It was shown that the probability density function of the added noises should have zero measure for the set of zero-points and the relative probability density in the considered adjacent sets should be upper bounded by a positive constant (see the conditions in Theorem 3.1 of \cite{HeACC17} for accurate mathematical descriptions on the conditions). In addition, Laplace distribution was shown to satisfy the given conditions. Hence, we follow the existing works to adopt noises drawn from Laplace distributions with a decaying parameter to mask the transmitted information.
\end{Remark}


\subsection{Analysis on the disagreement of estimates}
In this section, we provide a bound for the expectation of the absolute difference between $y_i(k)$ and the actual aggregate action $\bar{x}(k).$
The following lemma is given to support the quantification of the estimation error.
\begin{Lemma}\label{lemm}
Suppose that Assumption \ref{comm} is satisfied. Then, for each nonnegative integer $k,$
\begin{equation}
\mathbb{E}(|\mathbf{1}_N^T\mathbf{y}(k)-\mathbf{1}_N^T\mathbf{x}(k)|)\leq \frac{Nd(1-\bar{q}^{k})}{1-\bar{q}}.
\end{equation}
\end{Lemma}
\begin{proof}
By \eqref{cotn}, we see that
\begin{equation}\label{ex}
\mathbf{1}_N^T\mathbf{y}(k+1)=\mathbf{1}_N^T (\mathbf{y}(k)+\mathbf{w}(k))+\mathbf{1}_N^T(\mathbf{x}(k+1)-\mathbf{x}(k)).
\end{equation}
Hence,
\begin{equation}
\begin{aligned}
&\mathbf{1}_N^T\mathbf{y}(k+1)-\mathbf{1}_N^T\mathbf{x}(k+1)\\
=&\mathbf{1}_N^T (\mathbf{y}(k)-\mathbf{x}(k))+\mathbf{1}_N^T\mathbf{w}(k),
\end{aligned}
\end{equation}
and
\begin{equation}\label{expp}
\begin{aligned}
&|\mathbf{1}_N^T\mathbf{y}(k+1)-\mathbf{1}_N^T\mathbf{x}(k+1)|\\
\leq &|\mathbf{1}_N^T (\mathbf{y}(k)-\mathbf{x}(k))|+|\mathbf{1}_N^T\mathbf{w}(k)|.
\end{aligned}
\end{equation}
Taking expectations on both sides of \eqref{expp} gives
\begin{equation}
\begin{aligned}
&\mathbb{E}(|\mathbf{1}_N^T\mathbf{y}(k+1)-\mathbf{1}_N^T\mathbf{x}(k+1)|)\\
\leq & \mathbb{E}(|\mathbf{1}_N^T (\mathbf{y}(k)-\mathbf{x}(k))|)+\mathbb{E}(|\mathbf{1}_N^T\mathbf{w}(k)|)\\
\leq &\mathbb{E}(|\mathbf{1}_N^T (\mathbf{y}(0)-\mathbf{x}(0))|) +\sum_{j=0}^k \mathbb{E}(|\mathbf{1}_N^T\mathbf{w}(j)|).
\end{aligned}
\end{equation}
Noticing that $\mathbf{y}(0)=\mathbf{x}(0),$
\begin{equation}
\begin{aligned}
&\mathbb{E}(|\mathbf{1}_N^T\mathbf{y}(k+1)-\mathbf{1}_N^T\mathbf{x}(k+1)|)\\
\leq &\sum_{j=0}^k \mathbb{E}(|\mathbf{1}_N^T\mathbf{w}(j)|)= \frac{Nd (1-\bar{q}^{k+1})}{1-\bar{q}},
\end{aligned}
\end{equation}
where the last inequality is derived by utilizing $\mathbb{E}(|w_i(j)|)=d\bar{q}^j$ for $i\in\mathcal{V}.$
\end{proof}

\begin{Remark}
Actually, by mathematical induction, it can be obtained that
\begin{equation}
\mathbb{E}(\mathbf{1}_N^T\mathbf{y}(k)-\mathbf{1}_N^T\mathbf{x}(k))=0,
\end{equation}
for all nonnegative integer $k$. However, due to the effect of the added noises, we can only conclude that  $\mathbb{E}(|\mathbf{1}_N^T\mathbf{y}(k)-\mathbf{1}_N^T\mathbf{x}(k)|)$ is bounded by $\frac{Nd(1-\bar{q}^{k})}{1-\bar{q}}$ as indicated in Lemma \ref{lemm}. The bias of $\mathbb{E}(|\mathbf{1}_N^T\mathbf{y}(k)-\mathbf{1}_N^T\mathbf{x}(k)|)$ would result in a convergence error as indicated in the upcoming theorems.
\end{Remark}

Based on Lemma \ref{lemm}, the following theorem which establishes the estimation error bound can be obtained.
%

\begin{Theorem}\label{th3}
Suppose that Assumptions \ref{ass}-\ref{comm} are satisfied. Then, for each positive integer $k,$
\begin{equation}
\begin{aligned}
&\mathbb{E}(|y_i(k)-\bar{x}(k)|)\\
\leq &\frac{2\sqrt{N}(N-1)C_1}{N}\gamma^{k}+\frac{2(N-1)\sqrt{N} d\gamma}{(\gamma-\bar{q})N}(\gamma^{k}-\bar{q}^{k})\\
&+\frac{2 (N-1)\sqrt{N} C c  (\gamma^{k}-q^{k})}{(\gamma-q)N}+\frac{d (1-\bar{q}^{k})}{1-\bar{q}},
\end{aligned}
\end{equation}
where $i\in\mathcal{V},$ and $C_1=\max_{n\in\mathcal{V}}|y_n(0)|,$ $n\in\mathcal{V}.$
\end{Theorem}
\begin{proof}
For each positive integer $k$,
\begin{equation}
\begin{aligned}
&\mathbb{E}(|y_i(k)- \bar{x}(k)|)\\
= &\mathbb{E}(|y_i(k)-\frac{\mathbf{1}_N^T\mathbf{y}(k)}{N}+\frac{\mathbf{1}_N^T\mathbf{y}(k)}{N}-\frac{\mathbf{1}_N^T \mathbf{x}(k)}{N}|)\\
\leq & \frac{1}{N}\sum_{j=1,j\neq i}^N\mathbb{E}(|y_i(k)-y_j(k)|)\\
&+\frac{1}{N} \mathbb{E}(|\mathbf{1}_N^T\mathbf{y}(k)-\mathbf{1}_N^T\mathbf{x}(k)|).
\end{aligned}
\end{equation}
By \eqref{cotn}, we get that
\begin{equation}\label{ex22}
\begin{aligned}
&\mathbf{y}(k+1)=\mathcal{A}(\mathbf{y}(k)+\mathbf{w}(k))+\mathbf{x}(k+1)-\mathbf{x}(k)\\
=&\mathcal{A}^2 \mathbf{y}(k-1)+\mathcal{A}^2\mathbf{w}(k-1)+\mathcal{A} \mathbf{w}(k)\\
&+\mathcal{A}(\mathbf{x}(k)-\mathbf{x}(k-1))+\mathbf{x}(k+1)-\mathbf{x}(k).
\end{aligned}
\end{equation}
By repeating the above process, it can be obtained that
\begin{equation}\label{ex2}
\begin{aligned}
\mathbf{y}(k+1)=&\mathcal{A}^{k+1}\mathbf{y}(0)+\sum_{j=0}^k \mathcal{A}^{k+1-j}\mathbf{w}(j)\\
&+\sum_{j=1}^{k+1}\mathcal{A}^{k+1-j}(\mathbf{x}(j)-\mathbf{x}(j-1)).
\end{aligned}
\end{equation}
Hence, by \eqref{ex2},
\begin{equation}
\begin{aligned}
&|y_i(k+1)-y_j(k+1)|\\
\leq &\sum_{n=1}^N|[\mathcal{A}^{k+1}]_{in}-[\mathcal{A}^{k+1}]_{jn}||y_n(0)|\\
&+\sum_{n=1}^N\sum_{l=0}^k|[\mathcal{A}^{k+1-l}]_{in}-[\mathcal{A}^{k+1-l}]_{jn} | |w_n(l)|\\
&+\sum_{n=1}^N\sum_{l=1}^{k+1}|[\mathcal{A}^{k+1-l}]_{in}-[\mathcal{A}^{k+1-l}]_{jn}||x_n(l)-x_n(l-1)|.
\end{aligned}
\end{equation}
Note that
\begin{equation}
\begin{aligned}
&\sum_{n=1}^N\left|[\mathcal{A}^k]_{in}-\frac{1}{N}\right|\\
=&\sum_{n=1}^N\left|\left[\mathcal{A}^k-\frac{\mathbf{1}_N\mathbf{1}_N^T}{N}\right]_{in}\right|\leq \left|\left|\mathcal{A}^k-\frac{\mathbf{1}_N\mathbf{1}_N^T}{N}\right|\right|_{\infty}\\
\leq& \sqrt{N}\left|\left|\mathcal{A}^k-\frac{\mathbf{1}_N\mathbf{1}_N^T}{N}\right|\right|
\leq \sqrt{N}\gamma^k,
\end{aligned}
\end{equation}
where the last inequality is obtained by Lemma \ref{lemam1}.
%

Hence,
\begin{equation}\label{exptt}
\begin{aligned}
&|y_i(k+1)-y_j(k+1)|\\
\leq & 2\sqrt{N}C_1 \gamma^{k+1}+2\sqrt{N}C \sum_{l=1}^{k+1}\gamma^{k+1-l}\alpha_{l-1}\\
&+\sum_{n=1}^N\sum_{l=0}^k|[\mathcal{A}^{k+1-l}]_{in}-[\mathcal{A}^{k+1-l}]_{jn} | |w_n(l)|,
\end{aligned}
\end{equation}
where $C_1=\max_{n\in\mathcal{V}}|y_n(0)|$ as $|y_n(0)|$ for $n\in\mathcal{V}$ are bounded.

Taking expectations on both sides of \eqref{exptt} gives
\begin{equation}\label{exptt1}
\begin{aligned}
&\mathbb{E}(|y_i(k+1)-y_j(k+1)|)\\
\leq & 2\sqrt{N}C_1 \gamma^{k+1}+2\sqrt{N}C \sum_{l=1}^{k+1}\gamma^{k+1-l}\alpha_{l-1}\\
&+\sum_{n=1}^N\sum_{l=0}^k|[\mathcal{A}^{k+1-l}]_{in}-[\mathcal{A}^{k+1-l}]_{jn} | \mathbb{E}(|w_n(l)|).
\end{aligned}
\end{equation}

Recalling that $\mathbb{E}(|w_n(l)|)=d\bar{q}^l,$ for $n\in\mathcal{V},$
\begin{equation}
\begin{aligned}
&\sum_{n=1}^N\sum_{l=0}^k|[\mathcal{A}^{k+1-l}]_{in}-[\mathcal{A}^{k+1-l}]_{jn} | \mathbb{E}(|w_n(l)|)\\
=&\frac{2\sqrt{N} d\gamma}{\gamma-\bar{q}}(\gamma^{k+1}-\bar{q}^{k+1}).
\end{aligned}
\end{equation}
Similarly,
\begin{equation}
2\sqrt{N} C \sum_{l=1}^{k+1}\gamma^{k+1-l}\alpha_{l-1}
=\frac{2 \sqrt{N} C c (\gamma^{k+1}-q^{k+1})}{\gamma-q}.
\end{equation}


Hence,
\begin{equation}
\begin{aligned}
&\mathbb{E}(|y_i(k)-\bar{x}(k)|)\\
\leq &\frac{2\sqrt{N}(N-1)C_1}{N}\gamma^{k}+\frac{2(N-1)\sqrt{N} d\gamma}{(\gamma-\bar{q})N}(\gamma^{k}-\bar{q}^{k})\\
&+\frac{2 (N-1)\sqrt{N} C c (\gamma^{k}-q^{k})}{(\gamma-q)N}+\frac{d (1-\bar{q}^{k})}{1-\bar{q}},
\end{aligned}
\end{equation}
by further utilizing the results in Lemma \ref{lemm}.
\end{proof}

%

\subsection{Convergence analysis}
In this section, we establish the convergence results for the proposed method.
\begin{Theorem}\label{th1}
Suppose that Assumptions  \ref{a2}-\ref{comm} are satisfied.  Then,
\begin{equation}\label{hou}
\begin{aligned}
&\lim_{k\rightarrow \infty} \mathbb{E}(|| \mathbf{x}(k)-\mathbf{x}^*||^2)\\
\leq & C_2^2e^{-\frac{mc}{1-q}}+\frac{c^2NC^2}{1-q^2}+\Phi_1+\Phi_2,
\end{aligned}
\end{equation}
where
\begin{equation}\label{phi1}
\begin{aligned}
\Phi_1=&4\max_{i\in\mathcal{V}}\{l_i\}(N-1)\left(C_2+\frac{cC\sqrt{N}}{1-q}\right)\times\\
&\left(\frac{C_1c}{1-q\gamma}+\frac{Cc^2q}{(1-q\gamma)(1-q^2)}\right),
\end{aligned}
\end{equation}
 and
 \begin{equation}
 \begin{aligned}
 \Phi_2=&2\max_{i\in\mathcal{V}}\{l_i\} \left(C_2+\frac{cC\sqrt{N}}{1-q}\right)dcq \times\\
 &\left(\frac{2(N-1)\gamma}{(1-q\gamma)(1-\bar{q}q)}+\frac{\sqrt{N}}{(1-q)(1-\bar{q}q)}\right),
 \end{aligned}
 \end{equation}
where $C_2=||\mathbf{x}(0)-\mathbf{x}^*||.$
\end{Theorem}
\begin{proof}
By \eqref{cotn}, we can get that
\begin{equation}
\mathbf{x}(k+1)-\mathbf{x}^*=\mathbf{x}(k)-\mathbf{x}^*-\alpha_k [g_i(x_i(k),y_i(k))]_{vec}.
\end{equation}
Hence,
\begin{equation}
\begin{aligned}
&||\mathbf{x}(k+1)-\mathbf{x}^*||^2\\
=&||\mathbf{x}(k)-\mathbf{x}^*-\alpha_k [g_i(x_i(k),y_i(k))]_{vec}||^2\\
=&||\mathbf{x}(k)-\mathbf{x}^*||^2-2\alpha_k(\mathbf{x}(k)-\mathbf{x}^*)^T [g_i(x_i(k),y_i(k))]_{vec}\\
&+ \alpha_k^2 ||[g_i(x_i(k),y_i(k))]_{vec}||^2.
\end{aligned}
\end{equation}
Noticing that $g_i(x_i(k),y_i(k))$ is uniformly bounded by $C$, we can get that $||[g_i(x_i(k),y_i(k))]_{vec}||^2\leq NC^2.$ Hence
\begin{equation}
\alpha_k^2 ||[g_i(x_i(k),y_i(k))]_{vec}||^2\leq \alpha_k^2NC^2.
\end{equation}
 Moreover,
 \begin{equation}
 \begin{aligned}
 &-2\alpha_k(\mathbf{x}(k)-\mathbf{x}^*)^T [g_i(x_i(k),y_i(k))]_{vec}\\
 =& -2\alpha_k(\mathbf{x}(k)-\mathbf{x}^*)^T g(\mathbf{x}(k))\\
 &+2\alpha_k(\mathbf{x}(k)-\mathbf{x}^*)^T (g(\mathbf{x}(k))-[g_i(x_i(k),y_i(k))]_{vec})\\
 \leq & -2\alpha_km ||\mathbf{x}(k)-\mathbf{x}^*||^2\\
 &+2\alpha_k(\mathbf{x}(k)-\mathbf{x}^*)^T (g(\mathbf{x}(k))-[g_i(x_i(k),y_i(k))]_{vec}).
 \end{aligned}
 \end{equation}
Noticing that $g_i(\mathbf{x}(k))$ is globally Lipschitz with constant $l_i,$ we get that
\begin{equation}
\begin{aligned}
&||g(\mathbf{x}(k))-[g_i(x_i(k),y_i(k))]_{vec}||\\
\leq& \max_{i\in\mathcal{V}}\{l_i\}||\mathbf{y}(k)-\mathbf{1}_N\bar{x}(k)||.
\end{aligned}
\end{equation}

Therefore,
\begin{equation}
 \begin{aligned}
 &-2\alpha_k(\mathbf{x}(k)-\mathbf{x}^*)^T [g_i(x_i(k),y_i(k))]_{vec}\\
\leq & -2\alpha_km ||\mathbf{x}(k)-\mathbf{x}^*||^2\\
&+2\alpha_k\max_{i\in\mathcal{V}}\{l_i\}||\mathbf{x}(k)-\mathbf{x}^*||||\mathbf{y}(k)-\mathbf{1}_N\bar{x}(k)||,
\end{aligned}
\end{equation}
and
\begin{equation}
\begin{aligned}
&||\mathbf{x}(k+1)-\mathbf{x}^*||^2\\
\leq &||\mathbf{x}(k)-\mathbf{x}^*||^2-2\alpha_km ||\mathbf{x}(k)-\mathbf{x}^*||^2+\alpha_k^2NC^2\\
&+2\alpha_k\max_{i\in\mathcal{V}}\{l_i\}||\mathbf{x}(k)-\mathbf{x}^*||||\mathbf{y}(k)-\mathbf{1}_N\bar{x}(k)||.
\end{aligned}
\end{equation}

Moreover, as $\mathbf{x}(k+1)=\mathbf{x}(k)-\alpha_k[g_i(x_i(k),y_i(k))]_{vec}$,
\begin{equation}
\begin{aligned}
||\mathbf{x}(k+1)-\mathbf{x}^*||\leq &||\mathbf{x}(k)-\mathbf{x}^*|| +\alpha_k \sqrt{N} C\\
\leq & ||\mathbf{x}(0)-\mathbf{x}^*||+\frac{cC\sqrt{N}}{1-q}.
\end{aligned}
\end{equation}
Let $C_2=||\mathbf{x}(0)-\mathbf{x}^*||$, we have
\begin{equation}\label{expe}
\begin{aligned}
&||\mathbf{x}(k+1)-\mathbf{x}^*||^2\\
\leq &||\mathbf{x}(k)-\mathbf{x}^*||^2-2\alpha_km ||\mathbf{x}(k)-\mathbf{x}^*||^2+\alpha_k^2NC^2\\
&+2\alpha_k\max_{i\in\mathcal{V}}\{l_i\}\left(C_2+\frac{cC\sqrt{N}}{1-q}\right)||\mathbf{y}(k)-\mathbf{1}_N\bar{x}(k)||.
\end{aligned}
\end{equation}

Taking expectations on both sides of \eqref{expe} gives
\begin{equation}\label{exp}
\begin{aligned}
&\mathbb{E}(||\mathbf{x}(k+1)-\mathbf{x}^*||^2)\\
\leq & (1-2\alpha_km)\mathbb{E}(||\mathbf{x}(k)-\mathbf{x}^*||^2)+\alpha_k^2NC^2\\
&+2\alpha_k\max_{i\in\mathcal{V}}\{l_i\}\left(C_2+\frac{cC\sqrt{N}}{1-q}\right)\times\\
&\mathbb{E}(||\mathbf{y}(k)-\mathbf{1}_N\bar{x}(k)||).
\end{aligned}
\end{equation}
By Theorem \ref{th3},
\begin{equation}
\begin{aligned}
&2\alpha_k\max_{i\in\mathcal{V}}\{l_i\}\left(C_2+\frac{cC\sqrt{N}}{1-q}\right)\mathbb{E}(||\mathbf{y}(k)-\mathbf{1}_N\bar{x}(k)||)\\
\leq &2\alpha_k\max_{i\in\mathcal{V}}\{l_i\}\left(C_2+\frac{cC\sqrt{N}}{1-q}\right)(2(N-1)\gamma^{k}C_1\\
&+\frac{2(N-1) d\gamma}{\gamma-\bar{q}}(\gamma^{k}-\bar{q}^{k})+\frac{2 (N-1) C c (\gamma^{k}-q^{k})}{\gamma-q}\\
&+\frac{\sqrt{N}d (1-\bar{q}^{k})}{1-\bar{q}}).
\end{aligned}
\end{equation}

Therefore,
\begin{equation}\label{expp}
\begin{aligned}
&\mathbb{E}(||\mathbf{x}(k+1)-\mathbf{x}^*||^2)\\
\leq & \prod_{j=0}^k(1-2\alpha_jm)\mathbb{E}(||\mathbf{x}(0)-\mathbf{x}^*||^2)+\sum_{j=0}^k\alpha_j^2NC^2\\
&+\sum_{j=0}^k 2\alpha_j\max_{i\in\mathcal{V}}\{l_i\}\left(C_2+\frac{cC\sqrt{N}}{1-q}\right)\left(2(N-1)\gamma^{j}C_1\right.\\
&+\frac{2(N-1) d\gamma}{\gamma-\bar{q}}(\gamma^{j}-\bar{q}^{j})+\frac{2 (N-1) C c (\gamma^{j}-q^{j})}{\gamma-q}\\
&\left.+\frac{\sqrt{N}d (1-\bar{q}^{j})}{1-\bar{q}}\right).
\end{aligned}
\end{equation}

Let $k\rightarrow \infty$, then
\begin{equation}\label{expp1}
\begin{aligned}
&\lim_{k\rightarrow \infty}\mathbb{E}(||\mathbf{x}(k)-\mathbf{x}^*||^2)\\
\leq & \prod_{j=0}^{\infty}(1-2\alpha_jm)\mathbb{E}(||\mathbf{x}(0)-\mathbf{x}^*||^2)+\sum_{j=0}^{\infty}\alpha_j^2NC^2\\
&+\sum_{j=0}^{\infty} 2\alpha_j\max_{i\in\mathcal{V}}\{l_i\}\left(C_2+\frac{cC\sqrt{N}}{1-q}\right)\left(2(N-1)\gamma^{j}C_1\right.\\
&+\frac{2(N-1) d\gamma}{\gamma-\bar{q}}(\gamma^{j}-\bar{q}^{j})+\frac{2 (N-1) C c (\gamma^{j}-q^{j})}{\gamma-q}\\
&\left.+\frac{\sqrt{N}d (1-\bar{q}^{j})}{1-\bar{q}}\right),
\end{aligned}
\end{equation}
in which
\begin{equation}
\begin{aligned}
&\sum_{j=0}^{\infty}2\alpha_j\max_{i\in\mathcal{V}}\{l_i\}\left(C_2+\frac{cC\sqrt{N}}{1-q}\right)\left(2(N-1)\gamma^{j}C_1\right)\\
= & 4(N-1) C_1 \max_{i\in\mathcal{V}}\{l_i\} \left(C_2+\frac{cC\sqrt{N}}{1-q}\right) \frac{c}{1-q \gamma},
\end{aligned}
\end{equation}
and
\begin{equation}
\begin{aligned}
&\sum_{j=0}^{\infty} 2\alpha_j\max_{i\in\mathcal{V}}\{l_i\}(C_2+\frac{cC\sqrt{N}}{1-q})\frac{2(N-1) d\gamma}{\gamma-\bar{q}}(\gamma^{j}-\bar{q}^{j})\\
= & 4(N-1)\max_{i\in\mathcal{V}}\{l_i\} d\gamma \left(C_2+\frac{cC\sqrt{N}}{1-q}\right)\frac{cq}{(1-q\gamma)(1-\bar{q}q)}.
\end{aligned}
\end{equation}

In addition,
\begin{equation}
\begin{aligned}
&\sum_{j=0}^{\infty}2\alpha_j\max_{i\in\mathcal{V}}\{l_i\}\left(C_2+\frac{cC\sqrt{N}}{1-q}\right)\frac{2 (N-1) C c (\gamma^{j}-q^{j})}{\gamma-q}\\
&=4\max_{i\in\mathcal{V}}\{l_i\} \left(C_2+\frac{cC\sqrt{N}}{1-q}\right)\left(\frac{(N-1)Cc^2q}{(1-q\gamma)(1-q^2)}\right),
\end{aligned}
\end{equation}
and
\begin{equation}
\begin{aligned}
&\sum_{j=0}^\infty 2\alpha_j\max_{i\in\mathcal{V}}\{l_i\}\left(C_2+\frac{cC\sqrt{N}}{1-q}\right)\frac{\sqrt{N}d (1-\bar{q}^{j})}{1-\bar{q}}\\
=& 2 \max_{i\in\mathcal{V}}\{l_i\}\left(C_2+\frac{cC\sqrt{N}}{1-q}\right) \frac{\sqrt{N}dcq}{(1-q)(1-q\bar{q})}.
\end{aligned}
\end{equation}
Rearranging these terms gives
\begin{equation}
\begin{aligned}
&\lim_{k\rightarrow \infty} \mathbb{E}(|| \mathbf{x}(k)-\mathbf{x}^*||^2)\\
\leq & \prod_{j=0}^{\infty} (1-\alpha_jm) \mathbb{E}(||\mathbf{x}(0)-\mathbf{x}^*||^2)\\
&+\frac{c^2NC^2}{1-q^2}+\Phi_1+\Phi_2\\
\leq & C_2^2e^{-\frac{mc}{1-q}}+\frac{c^2NC^2}{1-q^2}+\Phi_1+\Phi_2,
\end{aligned}
\end{equation}
in which the last inequality is derived by utilizing $1-\alpha_jm\leq e^{-\alpha_jm}$ and $\mathbb{E}(||\mathbf{x}(0)-\mathbf{x}^*||^2)\leq C_2^2.$ Moreover,
\begin{equation}
\begin{aligned}
\Phi_1=&4\max_{i\in\mathcal{V}}\{l_i\}(N-1)\left(C_2+\frac{cC\sqrt{N}}{1-q}\right)\times\\
&\left(\frac{C_1c}{1-q\gamma}+\frac{Cc^2q}{(1-q\gamma)(1-q^2)}\right),
\end{aligned}
\end{equation}
 and
 \begin{equation}
 \begin{aligned}
 \Phi_2=&2\max_{i\in\mathcal{V}}\{l_i\} \left(C_2+\frac{cC\sqrt{N}}{1-q}\right)dcq \times\\
 &\left(\frac{2(N-1)\gamma}{(1-q\gamma)(1-\bar{q}q)}+\frac{\sqrt{N}}{(1-q)(1-\bar{q}q)}\right).
 \end{aligned}
 \end{equation}
\end{proof}
\begin{Remark}
From the above analysis, it is clear that $\lim_{k\rightarrow \infty} \mathbb{E}(|| \mathbf{x}(k)-\mathbf{x}^*||^2)$ is bounded by $C_2^2e^{-\frac{mc}{1-q}}+\frac{c^2NC^2}{1-q^2}+\Phi_1+\Phi_2,$ where $C_2^2e^{-\frac{mc}{1-q}}$ is resulted from the initial error. Moreover, $\frac{c^2NC^2}{1-q^2}$ and $\Phi_1$ depend on stepsize selection. In addition, $\Phi_2$ is resulted from the added noises $w_i(k),i\in\mathcal{V}$.
\end{Remark}

\subsection{Differential privacy}
In this section, we show that the proposed method is $\epsilon$-differentially private. To analyze the privacy level for the proposed Nash equilibrium seeking strategy, we consider two adjacent function sets  $F^{(1)}=\{f_i^{(1)}\}_{i=1}^N$, $F^{(2)}=\{f_i^{(2)}\}_{i=1}^N$ where $f_i^{(1)}=f_i^{(2)},\forall i\neq i_0,$ and $f_i^{(1)}\neq f_i^{(2)}$ for some $i_0\in\mathcal{V}.$
Note that if $\nabla_i f_i^{(1)}=\nabla_i f_i^{(2)}$ for $i=i_0,$ the two sequences generated by the proposed method under $F^{(1)}$ and $F^{(2)}$ would be the same by enforcing the added noises to be the same, indicating that the proposed method is of complete privacy. Hence, in the rest, we only consider the case in which $\nabla_i f_i^{(1)}\neq \nabla_i f_i^{(2)}$ for $i=i_0.$ In addition, the worst case scenario is considered, i.e., the adversary knows $\mathcal{A},\mathbf{x}(0),\{f_i\}_{i\neq i_0},\alpha_k$ and the distributions of the random variables. Moreover, the observations are  $\mathcal{O}=\{\mathcal{O}_k\}_{k=0}^{\infty},$ where $\mathcal{O}_k=\{\mathbf{p}(k)\}.$


Then, the following theorem can be derived.

\begin{Theorem}\label{the2}
Suppose that Assumption \ref{ass} is satisfied. Then, the proposed method in \eqref{al1}-\eqref{al2} is $\epsilon$-differentially private, where
\begin{equation}\label{priv}
\epsilon=\frac{2 cC\bar{q}}{d(\bar{q}-q)}.
\end{equation}
\end{Theorem}

\begin{proof}
For any fixed initial state $\mathbf{x}(0)$, we see from \eqref{al1}-\eqref{al2} that
\begin{equation}
x_i(1)=x_i(0)-\alpha_0g_i(x_i(0),y_i(0)),
\end{equation}
and hence $x_i(1),$ and $x_i(1)-x_i(0)$ are fixed. Moreover,
\begin{equation}
y_i(1)=\sum_{j=1}^{N}a_{ij}p_j(0)+x_i(1)-x_i(0),
\end{equation}
and hence if $p_j(0)$ for $j\in\mathcal{V}$ are fixed, we get that $w_l(0)$ and $y_i(1)$ are fixed. Repeating the above analysis, we get that for any set of objective functions, if we fix the observation sequence, then, there exists a unique sequence $\mathbf{x}(k),\mathbf{y}(k),\mathbf{w}(k)$ for $k\in\{0,1,2,\cdots\}$ that can generate the sequence of observation. Hence, under the set of objective functions $F^{(l)},l\in\{1,2\}$, there exists a bijective mapping from the noise sequence to the set of observations. For notational convenience, denote the mapping as $\Omega^{(l)}(\mathbf{w}),l\in\{1,2\},$ respectively.

Moreover, for presentation clarity, we denote the sequence generated by the proposed method under $F^{(1)}$ as
\begin{equation}\label{a5l2}
\begin{aligned}
x_i^{(1)}(k+1)=&x_i^{(1)}(k)-\alpha_kg_i^{(1)}(x_i^{(1)}(k),y_i^{(1)}(k)),\\
y_i^{(1)}(k+1)=&\sum_{j=1}^{N}a_{ij}p_j^{(1)}(k)+x_i^{(1)}(k+1)-x_i^{(1)}(k).
\end{aligned}
\end{equation}
Correspondingly, the sequence generated by the proposed method under $F^{(2)}$ is given by

\begin{equation}\label{a5l3}
\begin{aligned}
x_i^{(2)}(k+1)&=x_i^{(2)}(k)-\alpha_kg_i^{(2)}(x_i^{(2)}(k),y_i^{(2)}(k)),\\
y_i^{(2)}(k+1)&=\sum_{j=1}^{N}a_{ij}p_j^{(2)}(k)+x_i^{(2)}(k+1)-x_i^{(2)}(k).
\end{aligned}
\end{equation}
As $x_i^{(1)}(0)=x_i^{(2)}(0),y_i^{(1)}(0)=y_i^{(2)}(0),$ it can be easily obtained that for $i\neq i_0,$
\begin{equation}
x_i^{(1)}(k)=x_i^{(2)}(k),y_i^{(1)}(k)=y_i^{(k)}(k),
\end{equation}
 for all nonnegative integer $k,$ given that the two function sets generate the same observations.
Therefore, to ensure that $y_i^{(1)}(k)+w_i^{(1)}(k)=y_i^{(2)}(k)+w_i^{(2)}(k),$ we only need to enforce that,
\begin{equation}\label{dp1}
w_i^{(1)}(k)=w_i^{(2)}(k),
\end{equation}
for $i\neq i_0.$

Moreover, for $i=i_0$, we enforce that
\begin{equation}\label{dp2}
\Delta w_i(k)=-\Delta y_i(k),
\end{equation}
in which $\Delta w_i(k)=w_i^{(1)}(k)-w_i^{(2)}(k),\Delta y_i(k)=y_i^{(1)}(k)-y_i^{(2)}(k)$ and
\begin{equation}\label{dp3}
\begin{aligned}
\Delta y_i(k+1)&=\Delta x_i(k+1)-\Delta x_i(k)\\
&=- \alpha_k \Delta g_i(k),
\end{aligned}
\end{equation}
where $\Delta g_i(k)=g_i^{(1)}(x^{(1)}_i(k),y_i^{(1)}(k))-g_i^{(2)}(x_i^{(2)}(k),y_i^{(2)}(k)),
$
and $\Delta x_i(k)=x_i^{(1)}(k)-x_i^{(2)}(k).$

Let $\mathbf{w}^{(l)}(k)=[w_1^{(l)}(k),w_2^{(l)}(k),\cdots,w_N^{(l)}(k)]^T$ and $\mathbf{w}^{(l)}=\{\mathbf{w}^{(l)}(k)\}_{k=0}^{\infty}$  for $l\in\{1,2\}$. Moreover, let $\mathcal{O}^{(l)}$ be the sequence of observations under the function set $F^{(l)}.$ According to \eqref{dp1}-\eqref{dp3}, we know that for each $\mathbf{w}^{(1)}$, there exists a unique $\mathbf{w}^{(2)}$ such that $\mathcal{O}^{(1)}=\mathcal{O}^{(2)}.$ Let $\mathcal{B}(\cdot)$ be a mapping such that $\mathbf{w}^{(2)}=\mathcal{B}(\mathbf{w}^{(1)})$ if and only if $\mathcal{O}^{(1)}=\mathcal{O}^{(2)}.$ Then, it is clear that $\mathcal{B}(\cdot)$ is bijective.
Let $\Gamma^{(l)}=\{\mathbf{w}^{(l)}|\Omega^{(l)}(\mathbf{w})\in\mathcal{O}\},$ then, following the analysis in \cite{DingCDC18}, it can be obtained that
\begin{equation}
\frac{\mathbb{P}\{F^{(1)}|\mathcal{O}\}}{\mathbb{P}\{F^{(2)}|\mathcal{O}\}}=
\frac{\int_{\Gamma^{(1)}}\prod_{i=1}^N \prod_{k=0}^{\infty} \mathcal{L}(w_i^{(1)}(k),\theta_k)d\mathbf{w}^{(1)} }{\int_{\Gamma^{(2)}}\prod_{i=1}^N \prod_{k=0}^{\infty} \mathcal{L}(\mathcal{B}(w_i^{(1)}(k)),\theta_k)d\mathbf{w}^{(1)}},
\end{equation}
in which
\begin{equation}
\begin{aligned}
&\frac{\prod_{i=1}^N \prod_{k=0}^{\infty} \mathcal{L}(w_i^{(1)}(k),\theta_k)}{\prod_{i=1}^N \prod_{k=0}^{\infty} \mathcal{L}(\mathcal{B}(w_i^{(1)}(k)),\theta_k)}\\
\leq &\prod_{i=1}^N \prod_{k=0}^{\infty} e^{\frac{|p_i^{(1)}(k)-y_i^{(1)}(k)-(p_i^{(2)}(k)-y_i^{(2)}(k))|}{\theta_k}}\\
=&\prod_{k=0}^{\infty}e^{\frac{|\Delta w_{i_0}(k)|}{\theta(k)}}\leq e^{\sum_{k=0}^{\infty}\frac{|\Delta y_{i_0}(k)|}{\theta(k)}}.
\end{aligned}
\end{equation}
By the boundedness of the gradient value, we get that
$|\Delta y_{i_0}(k)|\leq 2 C \alpha_k.$

Therefore,
\begin{equation}
e^{\sum_{k=0}^{\infty}\frac{|\Delta y_{i_0}(k)|}{d\bar{q}^k}}\leq e^ {\sum_{k=0}^{\infty}\frac{{2 C  \alpha_k}}{d\bar{q}^k}}.
\end{equation}

By further noticing that $\bar{q}\in(q,1),$ we get that
\begin{equation}
e^{\sum_{k=0}^{\infty}\frac{|\Delta y_{i_0}(k)|}{d\bar{q}^k}}\leq e^{\frac{2Cc\bar{q}}{d(\bar{q}-q)}}.
\end{equation}
Note that $i_0\in \mathcal{V}$ stands for any player engaged in the game, we can conclude that the privacy level of the whole system is $\epsilon=\frac{2Cc\bar{q}}{d(\bar{q}-q)},$ i.e., the proposed method is $\epsilon$-differentially private.
\end{proof}

\begin{Remark}
In many existing works on distributed optimization and Nash equilibrium seeking, the decaying stepsize is required to satisfy
\begin{equation}
\sum_{k=0}^{\infty} \alpha_k=\infty, \sum_{k=0}^{\infty} \alpha_k^2<\infty,
\end{equation}
to establish the convergence results (see, e.g., \cite{KoshalOR16}\cite{Ram}). However, in this paper, we adopt a decaying stepsize that satisfies
\begin{equation}
\sum_{k=0}^{\infty} \alpha_k<\infty.
\end{equation}
Note that this is required to establish the $\epsilon$-differential privacy of the proposed method, i.e., differential privacy is not achievable with not summable stepsize in the proposed method. Therefore, the convergence accuracy is sacrificed to some extent for differential privacy. Note that this is in accordance with the ``impossibility result for $0$-LAS message perturbing algorithms" in \cite{Nozari} and some other existing results on differential privacy (see, e.g., \cite{Huang15}). The tradeoffs between the convergence property and differential privacy will be discussed in more details later.
\end{Remark}

\begin{Remark}
In reality, communication channels are often subject to various kinds of noises during information dissemination (see e.g., \cite{LiTAC2010}-\cite{LongIJRNC} and the references therein). Therefore, if we can do some experiments to investigate the characteristics and properties of communication noises in real communication channels, it would be an interesting open question to study whether it is possible to employ the noises in the communication channels to achieve differentially private Nash equilibrium seeking or not.
\end{Remark}

\subsection{Tradeoffs between the accuracy and privacy level}
From \eqref{hou}, it can be seen that
\begin{equation}
\lim_{k\rightarrow \infty} \mathbb{E}(|| \mathbf{x}(k)-\mathbf{x}^*||^2)\leq D,
\end{equation}
where $D=C_2^2e^{-\frac{mc}{1-q}}+\frac{c^2NC^2}{1-q^2}+\Phi_1+\Phi_2$ and $\Phi_2$ is resulted from the added noises.

By Theorem \ref{the2}, $\epsilon=\frac{2 cC\bar{q}}{d(\bar{q}-q)}$. Therefore,
\begin{equation}
d=\frac{2 cC\bar{q}}{\epsilon(\bar{q}-q)}.
\end{equation}
%

To see the tradeoff between the convergence accuracy and the privacy level, the accuracy bound $D$ can be restated in terms of $\epsilon$ as:
\begin{equation}\label{hou}
\begin{aligned}
D= & C_2^2e^{-\frac{mc}{1-q}}+\frac{c^2NC^2}{1-q^2}+\Phi_1+\tilde{\Phi}_2,
\end{aligned}
\end{equation}
where $\Phi_1$ is defined in  \eqref{phi1} and,
 \begin{equation}
 \begin{aligned}
 \tilde{\Phi}_2=&2\max_{i\in\mathcal{V}}\{l_i\} \left(C_2+\frac{cC\sqrt{N}}{1-q}\right) \frac{2C c^2q\bar{q}}{\epsilon(\bar{q}-q)}\times\\
 &\left(\frac{2(N-1)\gamma}{(1-q\gamma)(1-\bar{q}q)}+\frac{\sqrt{N}}{(1-q)(1-\bar{q}q)}\right).
 \end{aligned}
 \end{equation}

From \eqref{hou}, it can be seen that with $q,\bar{q},c$ being fixed, $D=O(\frac{1}{\epsilon}),$ indicating that the accuracy of the proposed method becomes arbitrarily bad when the method is of complete privacy. Hence, there is a tradeoff between the convergence accuracy and the privacy level. However, the presented result preserves the following properties:
\begin{enumerate}
  \item For any given $\epsilon>0,$ $c,d,\bar{q},q$ can be tuned such that the proposed method is $\epsilon$-differentially private;
  \item If privacy is not concerned, then, for any given positive constant $D$, $c,\bar{q},q$ can be tuned such that $\lim_{k\rightarrow\infty} \mathbb{E}(||\mathbf{x}-\mathbf{x}^*||^2)\leq D.$
\end{enumerate}
The first property can be easily obtained by \eqref{priv}. The second property is derived following the subsequent observations.
\begin{itemize}
  \item For any bounded initial condition, $C_2^2e^{-\frac{mc}{1-q}}$ can be tuned to be arbitrarily small by adjusting $c$ and $q$ such that $\frac{c}{1-q}$ is sufficiently large;
  \item Denote $\frac{c}{1-q}=\zeta$. Then, for fixed and sufficiently large $\zeta$,
      \begin{equation}
      \frac{c}{1-q\gamma}=\frac{\zeta (1-q)}{1-q\gamma},
      \end{equation}
      and
      \begin{equation}
      \frac{c^2}{1-q^2}=\frac{\zeta^2(1-q)}{1+q}.
      \end{equation}
      Noticing that the partial derivatives of $\frac{\zeta (1-q)}{1-q\gamma}$ and$\frac{\zeta^2(1-q)}{1+q}$ with respect to $q$ are negative for $q\in(0,1)$ and $\lim_{q\rightarrow 1}\frac{\zeta (1-q)}{1-q\gamma}=0$, $\lim_{q\rightarrow 1}\frac{\zeta^2(1-q)}{1+q}=0,$ it can be concluded that $c$ and $q$ can be adjusted such that $\frac{c}{1-q\gamma}$ and $\frac{c^2}{1-q^2}$ are sufficiently small with fixed $\zeta$. Therefore, by such a tuning rule, $C_2^2e^{-\frac{mc}{1-q}}+\frac{c^2NC^2}{1-q^2}+\Phi_1$ can be adjusted to be arbitrarily small.
  \item $\tilde{\Phi}_2$ depends on the level of privacy $\epsilon$ and if privacy is not of concern, $\epsilon$ can be chosen to be sufficiently large such that $\tilde{\Phi}_2$ is sufficiently small.
\end{itemize}

From \eqref{hou}, we see that the accuracy of the proposed method depends on $c, \epsilon, \bar{q}$ and $q$. Hence, we write $D$ as $D(c,q,\bar{q},\epsilon).$ With fixed requirement of privacy level (i.e., $\epsilon$), one can tune $c,q,\bar{q}$ by solving
\begin{equation}\label{mini}
\begin{aligned}
&\min_{c,q,\bar{q}} \ \ \ \ D(c,q,\bar{q}).\\
&\text{s.t.} \ \ \ \ c>0,q\in(0,1),\bar{q}\in(q,1).
\end{aligned}
\end{equation}
Moreover, noticing that for any fixed $c$ and $q$,
\begin{equation}
\frac{\partial D}{\partial \bar{q}}<0,
\end{equation}
indicating that $D$ is monotonically decreasing with any fixed $c,q$. Hence, one can choose $\bar{q}$ to be sufficiently close to $1$ as $\bar{q}\in(q,1)$ and solve
\begin{equation}\label{mini1}
\begin{aligned}
&\min_{c,q} \ \ \ \ D(c,q)\\
&\text{s.t.} \ \ \ \ c>0,q\in(0,\bar{q}),
\end{aligned}
\end{equation}
to get the optimal $c,q.$

%


\section{Extensions to time-varying communication topologies}\label{time}
In the previous section, we suppose that the communication topology among the players is fixed, undirected and connected for presentation clarity. Actually, the proposed method can be slightly adapted to accommodate time-varying communication graphs that satisfy the following conditions:
\begin{Assumption}\label{as5}
There exists a positive integer $z$ such that $(\mathcal{V},\bigcup_{l=1}^z \mathcal{E}_{l+k})$ is connected for all nonnegative integer $k$, in which $\mathcal{G}_l=(\mathcal{V},\mathcal{E}_l)$ is the undirected communication graph at time $l$ and $\mathcal{E}_l$ is the corresponding edge set at time $l$.
\end{Assumption}

Note that in Section \ref{main}, the communication graph is supposed to be connected at each time instant $k=\{0,1,2,\cdots\}.$ However, in this section, the communication condition is generalized to be time-varying and it is only required that there exists a positive integer $z$ such that the joint graph $\bigcup_{l=1}^z \mathcal{G}_{l+k}$ is connected for nonnegative integer $k$.

\begin{Assumption}\label{as6}
Let $\mathcal{A}(l)$ be the weight matrix associated with $\mathcal{G}_l$ that satisfies:
\begin{enumerate}
  \item There exits a positive constant $\delta$ such that $a_{ii}(l)>\delta$, $a_{ij}(l)>\delta$ if $(j,i)\in\mathcal{E}_l$ and $a_{ij}(l)=0,$ if  $(j,i)\notin\mathcal{E}_l$;
  \item $\mathcal{A}(l)$ is doubly stochastic, i.e., $\mathbf{1}_N^T\mathcal{A}(l)=\mathbf{1}_N^T,$ $\mathcal{A}(l)\mathbf{1}_N=\mathbf{1}_N.$
\end{enumerate}
\end{Assumption}

Let $\Psi(k,s)=\mathcal{A}(k)\mathcal{A}(k-1)\cdots \mathcal{A}(s)$ where $k\geq s\geq 0.$ Then, under Assumptions \ref{as5}-\ref{as6}, the following lemma can be obtained.
\begin{Lemma}
\cite{NEDIC08} Let Assumptions \ref{as5}-\ref{as6} be satisfied. Then,
\begin{enumerate}
  \item $\lim_{k\rightarrow \infty} \Psi(k,s)=\frac{\mathbf{1}_N\mathbf{1}_N^T}{N},$ $\forall s\geq 0;$
  \item For $k\geq s \geq 0,$ $ |[\Psi(k,s)]_{ij}-\frac{1}{N}|\leq \theta \beta^{k+1-s},$
  where $\theta=(1-\frac{\delta}{4N^2})^{-2},$ $\beta=(1-\frac{\delta}{4N^2})^{\frac{1}{z}}.$
\end{enumerate}
\end{Lemma}

With time-varying communication topologies, the Nash equilibrium seeking strategy should be adapted as
\begin{equation}\label{cotn1}
\begin{aligned}
\mathbf{x}(k+1)&=\mathbf{x}(k)-\alpha_k [g_i(x_i(k),y_i(k))]_{vec}\\
\mathbf{y}(k+1)&=\mathcal{A}(k)\mathbf{p}(k)+\mathbf{x}(k+1)-\mathbf{x}(k).
\end{aligned}
\end{equation}

Note that the variations of the communication graph will only affect the convergence results but not the differential privacy of the proposed method. Hence, in the following, only the convergence results are presented.

First, following the proof of Theorem \ref{th3}, the subsequent result can be obtained.
\begin{Theorem}\label{col1}
Suppose that Assumptions \ref{as5}-\ref{as6} are satisfied. Then, for each positive integer $k,$
\begin{equation}
\begin{aligned}
 &\mathbb{E}(|y_i(k)-\bar{x}(k)|)\\
\leq &2(N-1)\theta \beta^{k}C_1+\frac{2(N-1)\theta d\beta}{\beta-\bar{q}}(\beta^{k}-\bar{q}^{k})\\
&+\frac{2\theta  (N-1) C c  (\beta^{k}-q^{k})}{\beta-q}+\frac{d (1-\bar{q}^{k})}{1-\bar{q}}.
\end{aligned}
\end{equation}
for $i\in\mathcal{V}.$
\end{Theorem}
\begin{Remark}
From Theorem \ref{col1}, it can be seen that compared with the results under fixed communication topologies, the time-varying communication topology would affect the estimation speed.
\end{Remark}

Moreover, following Theorem \ref{th1}, the subsequent convergence result can be obtained.
\begin{Theorem}\label{col2}
Suppose that Assumptions 1-3 and \ref{as5}-\ref{as6} are satisfied. Then,
\begin{equation}
\begin{aligned}
&\lim_{k\infty} \mathbb{E}(|| \mathbf{x}(k)-\mathbf{x}^*||^2)\\
\leq & C_2^2e^{-\frac{mc}{1-q}}+\frac{c^2NC^2}{1-q^2}+\bar{\Phi}_1+\bar{\Phi}_2,
\end{aligned}
\end{equation}
where
\begin{equation}
\begin{aligned}
\bar{\Phi}_1=&4(N-1)\sqrt{N}\theta \max_{i\in\mathcal{V}}\{l_i\}\left(C_2+\frac{cC\sqrt{N}}{1-q}\right)\\
&\left(\frac{C_1c}{1-q\beta}+\frac{Cc^2q}{(1-\beta q)(1-q^2)}\right),
\end{aligned}
\end{equation}
and
\begin{equation}
\begin{aligned}
\bar{\Phi}_2=&2\max_{i\in\mathcal{V}}\{l_i\} \left(C_2+\frac{cC\sqrt{N}}{1-q}\right)\sqrt{N}dcq\times\\ &\left(\frac{2(N-1)\theta\beta}{(1-q\beta)(1-\bar{q}q)}+\frac{1}{(1-q)(1-\bar{q}q)}\right).
\end{aligned}
\end{equation}
\end{Theorem}

\section{A Numerical example}\label{num_ex}

In this section, we consider the $5$-player energy consumption game studied in \cite{YEcyber17}. In the energy consumption game, player $i$'s objective function is given by
\begin{equation}
f_i(\mathbf{x})=(x_i-\hat{x}_i)^2+(0.04\sum_{i=1}^5x_i+5)x_i,
\end{equation}
in which $\hat{x}_1=50,\hat{x}_2=55,\hat{x}_3=60,\hat{x}_4=65,\hat{x}_5=70.$
As demonstrated in \cite{YEcyber17}, the game has a unique pure-strategy Nash equilibrium at
$\mathbf{x}^*=(41.5,46.4,51.3,56.2,61.1).$ In the following, fixed communication topologies and time-varying communication topologies will be considered successively.

\subsection{Nash equilibrium seeking under fixed communication topologies}\label{fix}
In the simulation, the players are supposed to communicate via a cycle depicted in Fig. \ref{comm1}.
\begin{figure}[!t]
\centering
\hspace{-30mm}\scalebox{0.5}{\includegraphics{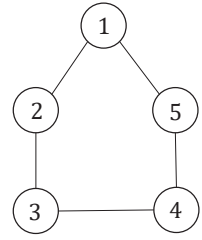}}
\vspace{-60mm}
\caption{The communication graph for the players.}\label{comm1}
\end{figure}
Correspondingly, the weight matrix is given as
\begin{equation}\nonumber
\mathcal{A}=\left[
               \begin{array}{ccccc}
                 0.5 & 0.2 & 0 & 0 & 0.3 \\
                 0.2 & 0.5 & 0.3 & 0 & 0 \\
                 0 & 0.3 & 0.5 & 0.2 & 0 \\
                 0 & 0 & 0.2 & 0.5 & 0.3 \\
                 0.3 & 0 & 0 & 0.3 & 0.4 \\
               \end{array}
             \right],
             \end{equation}
and $c=1,q=0.9,d=1,\bar{q}=0.99.$ Moreover, the proposed method is run for $2000$ times for the observation of the simulation results. Fig. \ref{simula1} shows the expectations of the players' squared Nash equilibrium seeking errors, i.e., $\mathbb{E}((x_i(k)-x_i^*)^2)$ for $i\in\{1,2,\cdots,5\}$ from which we see that the proposed method drives $\mathbb{E}((x_i(k)-x_i^*)^2)$ to a small neighborhood of zero.
Fig. \ref{simulas2} (1) plots $\mathbb{E}((y_i(k)-\bar{x}^*)^2)$ for $i\in\{1,2,\cdots,5\}$ generated by the proposed method and Fig. \ref{simulas2} (2) shows $\bar{y}_i(k)-\bar{x}^*$ for $i\in\{1,2,\cdots,5\}$, where $\bar{y}_i(k)$ denotes the averaged value of $y_i(k)$ for the $2000$ running times and $\bar{x}^*=\frac{1}{5}\sum_{i=1}^5x_i^*.$

\begin{figure}[!htp]
\centering
\scalebox{0.65}{\includegraphics{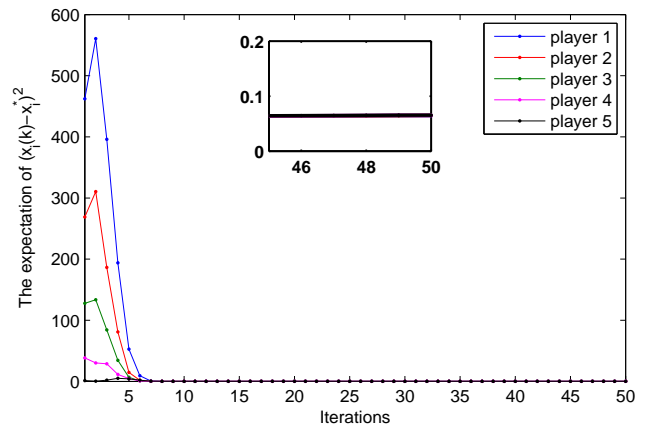}}
\caption{The expectations of the players' squared Nash equilibrium seeking errors, i.e., $\mathbb{E}((x_i(k)-x_i^*)^2),i\in\{1,2,\cdots,5\},$ generated by the proposed method in \eqref{al1}-\eqref{al2}.}\label{simula1}
\end{figure}

\begin{figure}[!htp]
\centering
\hspace{-8mm}\scalebox{0.65}{\includegraphics{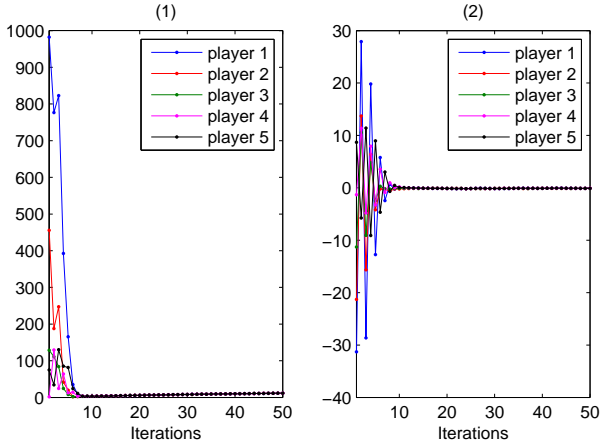}}
\caption{(1): The expectations of the players' squared estimation errors on $\bar{x}^*$, i.e., $\mathbb{E}((y_i(k)-\bar{x}^*)^2)$ for $i\in\{1,2,\cdots,5\}$; (2): The plots of $\bar{y}_i(k)-\bar{x}^*$ for $i\in\{1,2,\cdots,5\}$ generated by the proposed method in \eqref{al1}-\eqref{al2}.}\label{simulas2}
\end{figure}

To show the tradeoff between the convergence accuracy and $d$, we vary $d$ from $0$ to $3$ and observe the simulation results. Fig. \ref{simulas3} plots $\mathbb{E}(||\mathbf{x}(\infty)-\mathbf{x}^*||^2).$ Roughly speaking, the figure shows that as $d$ increases, the upper bound of $\mathbb{E}(||\mathbf{x}(\infty)-\mathbf{x}^*||^2)$ increases, which illustrates the tradeoff between the privacy level and the convergence accuracy.

\begin{figure}[!htp]
\centering
\hspace{-8mm}\scalebox{0.65}{\includegraphics{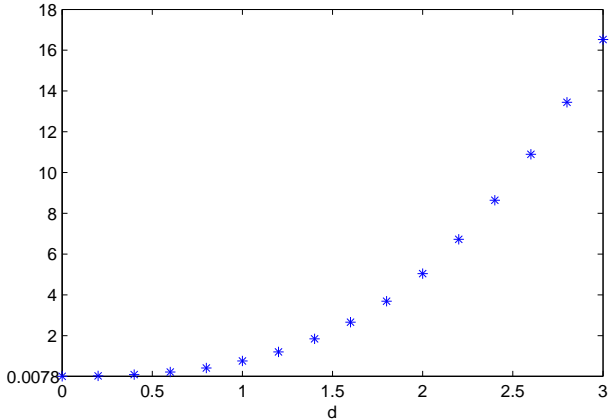}}
\caption{The plot of $\mathbb{E}(||\mathbf{x}(\infty)-\mathbf{x}^*||^2)$ with $d$ varying from $0$ to $3$.}\label{simulas3}
\end{figure}

To observe the data distributions generated by the proposed method, we fix $d=1,\bar{q}=0.99$ and run the proposed method for $20000$ times. Fig. \ref{simulas4} plots the data distributions of $x_i(\infty)$ and its corresponding fitted density functions for $i\in\{1,2,\cdots,5\}$. Fig. \ref{simulas5} shows the data distributions of $y_i(\infty)$ and the plots of the corresponding fitted density functions for players $1$-$5,$ respectively. From Figs. \ref{simulas4}-\ref{simulas5}, we see that the players' actions converge to a small neighborhood of the Nash equilibrium with a high probability under the given parameters.

\begin{figure}[!htp]
\centering
\scalebox{0.65}{\includegraphics{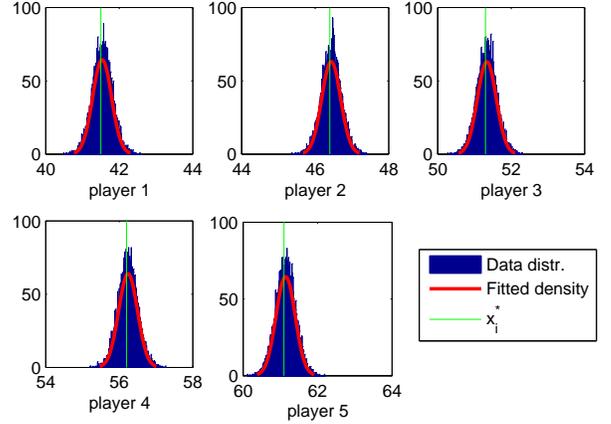}}
\caption{The plots of data distributions of $x_i(\infty)$ and the corresponding fitted density functions for players $1$-$5$, generated by the proposed method in \eqref{al1}-\eqref{al2}.}\label{simulas4}
\end{figure}

\begin{figure}[!htp]
\centering
\scalebox{0.65}{\includegraphics{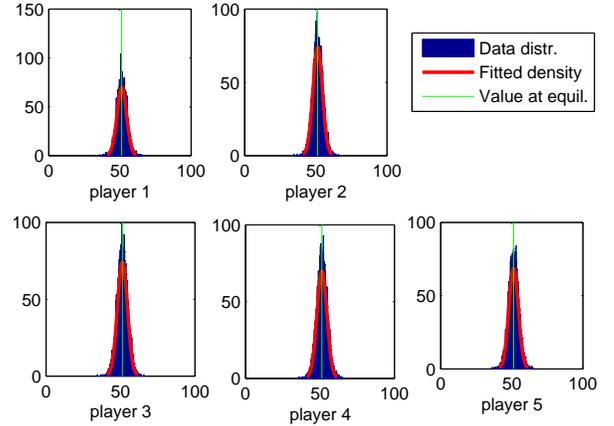}}
\caption{The plots of data distributions of $y_i(\infty)$ and the corresponding fitted density functions for players $1$-$5$, generated by the proposed method in \eqref{al1}-\eqref{al2}.}\label{simulas5}
\end{figure}

\subsection{Nash equilibrium seeking under time-varying communication topologies}
In the simulation, we suppose that the communication topology among the players switches between the two graphs depicted in Fig. \ref{comm2}. Correspondingly, we define
\begin{equation}\nonumber
\mathcal{A}(k)=\left[
                                             \begin{array}{ccccc}
                                               0.3 & 0.3 & 0 & 0 & 0.4 \\
                                               0.3 & 0.5 & 0.2 & 0 & 0 \\
                                               0 & 0.2 & 0.5 & 0 & 0.3 \\
                                               0 & 0 & 0 & 1 & 0 \\
                                               0.4 & 0 & 0.3 & 0 & 0.3 \\
                                             \end{array}
                                           \right],
\end{equation}
for $k\in\{0,2,4,6,\cdots\}$. In addition,
\begin{equation}\nonumber
\mathcal{A}(k)=\left[
                                             \begin{array}{ccccc}
                                               1 & 0 & 0 & 0 & 0 \\
                                               0 & 1 & 0 & 0 & 0 \\
                                               0 & 0 & 1 & 0 & 0 \\
                                               0 & 0 & 0 & 0.7 & 0.3 \\
                                               0 & 0 & 0 & 0.3 & 0.7 \\
                                             \end{array}
                                           \right],
\end{equation}
 for $k\in\{1,3,5,7\cdots\}$. For comparison convenience, the parameters are chosen as those in Section \ref{fix}. Moreover, the proposed method is run for $2000$ times for the observation of the simulation results. Fig. \ref{simula1_time_varying} plots $\mathbb{E}((x_i(k)-x_i^*)^2),i\in\{1,2,\cdots,5\},$ and the two sub-figures in Fig. \ref{simulas2_time_varying} depict  $\mathbb{E}((y_i(k)-\bar{x}^*)^2)$  and $\bar{y}_i(k)-\bar{x}^*$ for $i\in\{1,2,\cdots,5\},$ respectively. From Figs. \ref{simula1_time_varying}-\ref{simulas2_time_varying}, it can be seen that driven by the proposed method, $\mathbb{E}((x_i(k)-x_i^*)^2),$ $\mathbb{E}((y_i(k)-\bar{x}^*)^2)$ and $\bar{y}_i(k)-\bar{x}^*$ for $i\in\{1,2,\cdots,5\}$ would converge to a small neighborhood of zero. In addition, comparing Fig. \ref{simula1_time_varying}. (2) with Fig. \ref{simulas2}. (2), it is clear that  the players can achieve the aggregate estimation at a faster speed under the fixed communication topology given in Fig. \ref{comm1}.

\begin{figure}[!htp]
\centering
\scalebox{0.5}{\includegraphics{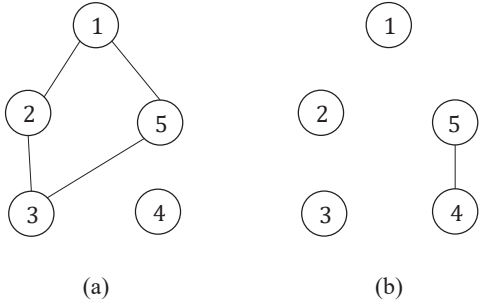}}
\vspace{-50mm}
\caption{(a) is the communication graph for the players at $k=\{0,2,4,6,\cdots\}$ and (b) is the communication graph for the players at $k=\{1,3,5,7,\cdots\}$.}\label{comm2}
\end{figure}

\begin{figure}[!htp]
\centering
\scalebox{0.56}{\includegraphics{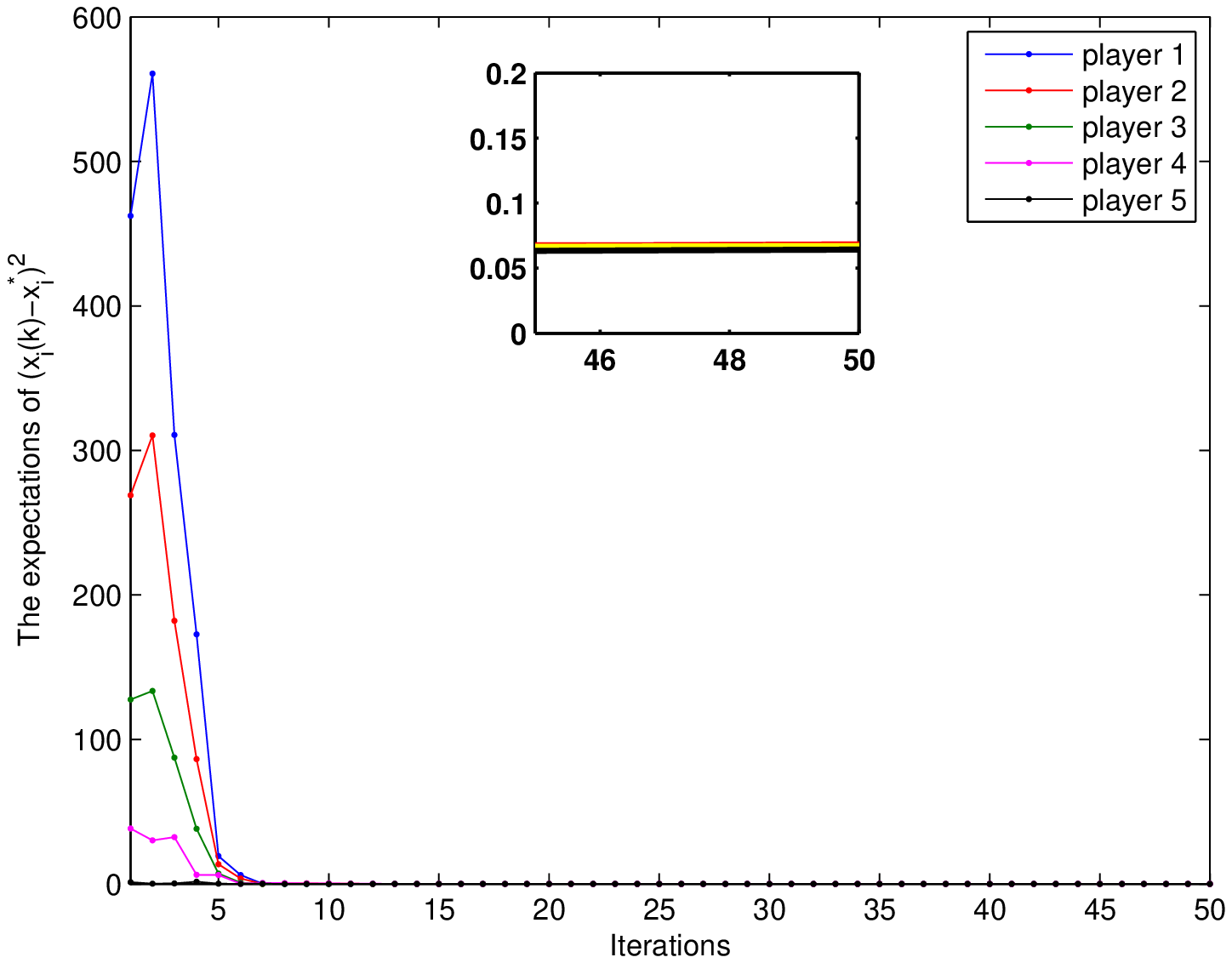}}
\caption{The expectations of the players' squared Nash equilibrium seeking errors, i.e., $\mathbb{E}((x_i(k)-x_i^*)^2),i\in\{1,2,\cdots,5\},$ generated by the proposed method in \eqref{cotn1}.}\label{simula1_time_varying}
\end{figure}

\begin{figure}[!htp]
\centering
\scalebox{0.56}{\includegraphics{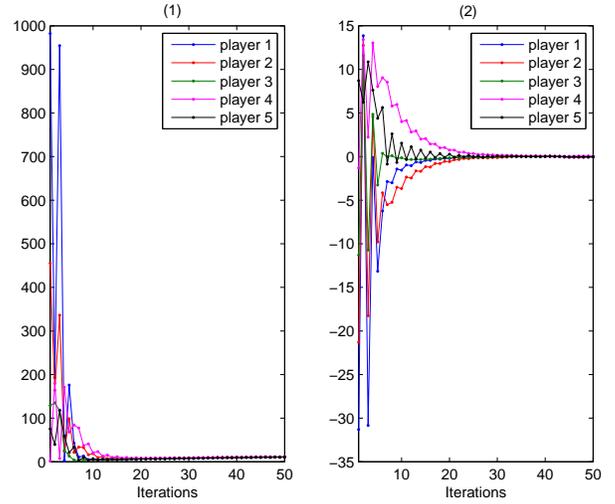}}
\caption{(1): The expectations of the players' squared estimation errors on $\bar{x}^*$, i.e., $\mathbb{E}((y_i(k)-\bar{x}^*)^2)$ for $i\in\{1,2,\cdots,5\}$; (2): The plots of $\bar{y}_i(k)-\bar{x}^*$ for $i\in\{1,2,\cdots,5\}$ generated by the proposed method in \eqref{cotn1}}\label{simulas2_time_varying}
\end{figure}

To show the tradeoff between the privacy level and the convergence accuracy under the time-varying communication topologies, we vary $d$ from $0$ to $3$. Correspondingly, the plot of $\mathbb{E}(||\mathbf{x}(\infty)-\mathbf{x}^*||^2)$ with different values of $d$ is given in Fig.  \ref{simulas3_timevying}, from which we see that $\mathbb{E}(||\mathbf{x}(\infty)-\mathbf{x}^*||^2)$ increases with $d$.

\begin{figure}[!htp]
\centering
\scalebox{0.54}{\includegraphics{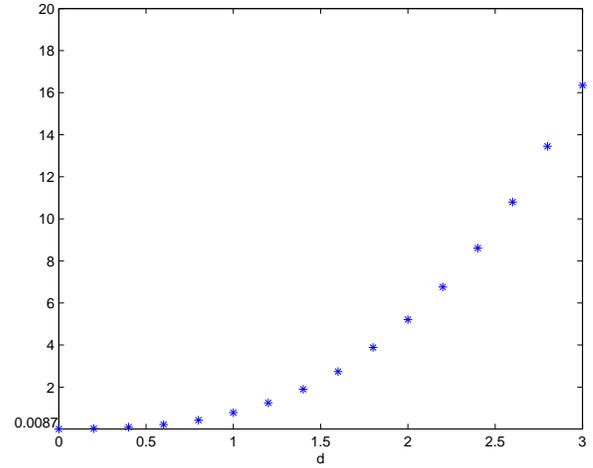}}
\caption{The plot of $\mathbb{E}(||\mathbf{x}(\infty)-\mathbf{x}^*||^2)$ with $d$ varying from $0$ to $3$.}\label{simulas3_timevying}
\end{figure}

\begin{figure}[!htp]
\centering
\hspace{-8mm}\scalebox{0.56}{\includegraphics{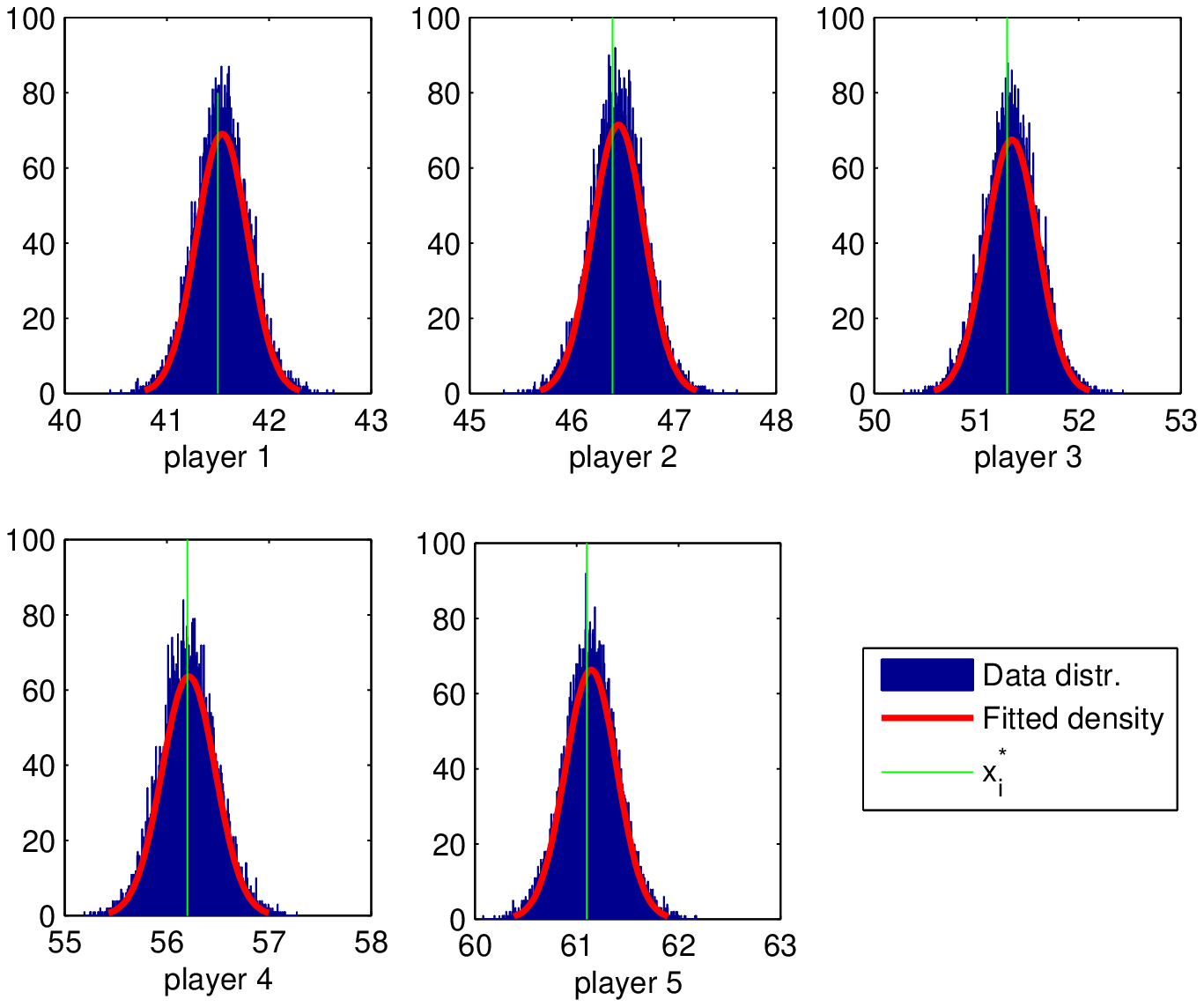}}
\caption{The plots of data distributions of $x_i(\infty)$ and the corresponding fitted density functions for players $1$-$5,$ generated by the proposed method in \eqref{cotn1}.}\label{simulas4_time_varying}
\end{figure}

\begin{figure}[!htp]
\centering
\hspace{-8mm}\scalebox{0.56}{\includegraphics{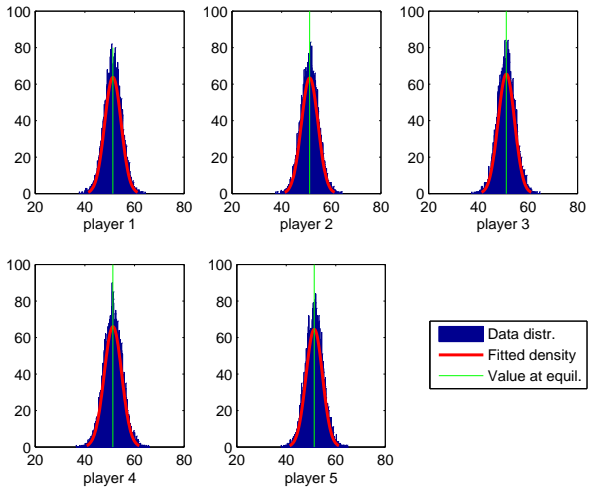}}
\caption{The plots of data distributions of $y_i(\infty)$ and  the corresponding fitted density functions for players $1$-$5,$ generated by the proposed method in \eqref{cotn1}.}\label{simulas5_time_varying}
\end{figure}

Likewise, to observe the data distributions generated by the proposed method, we fix $d=1,\bar{q}=0.99$ and run the proposed method for $20000$ times. The data distributions of $x_i(\infty)$ and $y_i(\infty)$ for $i\in\{1,2,\cdots,5\}$, together with their corresponding fitted density functions are plotted in Figs. \ref{simulas4_time_varying}-\ref{simulas5_time_varying}.  Figs. \ref{simulas4_time_varying}-\ref{simulas5_time_varying} show that the players' actions converge to a small neighborhood of the Nash equilibrium with a high probability under the given parameters.

\section{Conclusions}\label{conc}

This paper considers privacy-preservation in the distributed Nash equilibrium seeking problem for networked aggregative games. To estimate the averaged value of the players' actions, a dynamic average consensus protocol is employed in which the transmitted information is masked by independent random noises drawn from Laplace distributions.  The random noises are included for the protection of the players' objective functions. Moreover, with the estimated information, the gradient descent method with a decaying stepsize is implemented to optimize the players' objective functions. The convergence property as well as the privacy level of the proposed method are analytically investigated. It is shown that there is a tradeoff between the convergence accuracy and the privacy level. Fixed communication topologies and time-varying communication topologies are addressed successively in the paper. Privacy-preserving Nash equilibrium seeking for more general games (see e.g., the games considered in \cite{YETAC17}-\cite{YETAC19}) will be included in our future works.

\end{document}